\documentclass[11 pt]{article}

\usepackage[utf8]{inputenc}
\usepackage[english]{babel}   
\usepackage{graphicx}
\usepackage{lmodern}
\usepackage{amsmath}
\usepackage{tikz}
\usetikzlibrary{calc}
\usepackage{subfig}
\usepackage{longtable}
\usepackage{epsfig}
\usepackage{color}
\usepackage{setspace}
\usepackage[top=2cm,bottom=2.5cm,left=2.5cm,right=2.5cm]{geometry}
\usepackage{sectsty}
\usepackage{fancyhdr}
\usepackage{wasysym}
\usepackage[colorlinks=true,citecolor=blue]{hyperref}
\usepackage{caption}
\usetikzlibrary{patterns}
\usepackage{amsfonts}
\usepackage{amssymb}
\usepackage{amsthm}
\usepackage{algorithm}
\usepackage{algpseudocode}

\newtheorem{theorem}{Theorem} 
\newtheorem{definition}[theorem]{Definition}
\newtheorem{lemma}[theorem]{Lemma}
\newtheorem{proposition}[theorem]{Proposition}
\newtheorem{remark}[theorem]{Remark}
\newtheorem{hyp}[theorem]{Hypothesis}

\tikzset{
    right angle quadrant/.code={
        \pgfmathsetmacro\quadranta{{1,1,-1,-1}[#1-1]}     
        \pgfmathsetmacro\quadrantb{{1,-1,-1,1}[#1-1]}},
    right angle quadrant=1, 
    right angle length/.code={\def\rightanglelength{#1}},   
    right angle length=2ex, 
    right angle symbol/.style n args={3}{
        insert path={
            let \p0 = ($(#1)!(#3)!(#2)$) in     
                let \p1 = ($(\p0)!\quadranta*\rightanglelength!(#3)$), 
                \p2 = ($(\p0)!\quadrantb*\rightanglelength!(#2)$) in 
                let \p3 = ($(\p1)+(\p2)-(\p0)$) in  
            (\p1) -- (\p3) -- (\p2)
        }
    }
}

\title{$H^2$-conformal approximation of Miura surfaces}
\date{}
\author{\begin{minipage}{\textwidth}\centering
		Frédéric Marazzato \\
		\small{Department of Mathematics, Louisiana State University, Baton Rouge, LA 70803, USA}\\
   \small{email: \texttt{marazzato@lsu.edu}\\}
   \end{minipage}
   }

\begin{document}

\maketitle

\begin{abstract}
The Miura ori is a very classical origami pattern used in numerous applications in Engineering.
A study of the shapes that surfaces using this pattern can assume is still lacking.
A constrained nonlinear partial differential equation (PDE) that models the possible shapes that a periodic Miura tessellation can take in the homogenization limit has been established recently and solved only in specific cases.
In this paper, the existence and uniqueness of a solution to the unconstrained PDE is proved for general Dirichlet boundary conditions.
Then a $H^2$-conforming discretization is introduced to approximate the solution of the PDE coupled to a Newton method to solve the associated discrete problem.
A convergence proof for the method is given as well as a convergence rate.
Finally, numerical experiments show the robustness of the method and that non trivial shapes can be achieved using periodic Miura tessellations.
\end{abstract}

\textit{Keywords: Origami, nonlinear elliptic equation, Kinematics of deformation.}

\textit{AMS Subject Classification: 35J66, 65N12, 65N15, 65N30 and 74A05.}

\section{Introduction}
Origami inspired structures are used for multiple engineering applications.
A classical example is solar panels for satellites \cite{hernandez2019active}.
Indeed, the panels can be folded along the crease lines into very compact structures, easily stored in a rocket, and then unfolded into very wide panels in space.
More recently, origami inspired structures have gained attention as a mean to produce materials with a negative Poisson ratio \cite{lang2009origami} and metamaterials \cite{wickeler2020novel}.
The science behind origami is also used to fold airbags for optimal deployment. \cite{lang2009origami}.

The Miura ori or Miura tessellation \cite{miura1969proposition} is a well-known type of origami-inspired tessellation that has drawn a lot of attention over the years.
Miura tessellations have often been deemed useful when unfolded into a flat plane.
But, they can also achieve many non-planar shapes when partially unfolded which has recently allowed new applications \cite{liu2015deformation,wickeler2020novel}.
However, a lot remains unknown regarding the modeling of the shapes that Miura tessellations can take.
Indeed, simulating their exact shape through mechanical modeling is computationally involved due to the large number of degrees of freedom (dofs) considered.
The foldability of origami gives rise to notoriously difficult computational problems \cite{akitaya2018rigid}.
Only a few periodic cells can be simulated as for instance in \cite{wickeler2020novel}.
In \cite{schenk2013geometry,wei2013geometric}, the authors managed to prove that the in-plane and out-of-plane Poisson ratios of Miura tessellations are equal in norm but of opposite signs.
But that does not provide a way to compute Miura tessellations.

To remedy that issue, \cite{nassar2017curvature,lebee2018fitting} introduced a homogenization process leading to a set of equations that describe the shapes that Miura tessellations can fit in the limit $\frac{r}{R} \to 0$, where $r$ is the size of the pattern and $R$ is the global size of the structure.
The resulting equations describe parametric surfaces that are no longer discrete but continuous.
The main advantage of the homogenization approach is that it greatly reduces the computational cost of simulating Miura surfaces as one does not need to take into account all the dofs stemming from each individual polygon any longer.
Using this approach, the authors managed to determine all axisymmetric Miura surfaces.
They also built an algorithm to produce some non-planar Miura tessellations but it fails in certain situations.
The homogenization process produces a nonlinear elliptic PDE, as described in \cite{lebee2018fitting}, that has remained unsolved.
We believe that solving and providing a systematic and robust method to compute Miura surfaces will allow the exploration of the possible shapes that can be created with Miura tessellations.
It might also help shed some light on determining what surfaces can Miura tessellations fit, which, to the best of our knowledge, remains unknown.

In Section \ref{sec:continuous}, under regularity assumptions on the boundary conditions and the domain of the parametric surfaces, existence and uniqueness of solutions of the equation are proved.
$\mathcal{C}^{2,\alpha}$ regularity of the parametric surfaces is also proved in the process.
In Section \ref{sec:discrete}, a $H^2$-conforming finite element method (FEM) coupled to a Newton method is introduced to approximate the solution of the elliptic equation.
Subsequently, a first order convergence rate in $H^2$-norm is proved for the FEM approximation.
In Section \ref{sec:tests}, the convergence rate is verified on an analytical solution and then several non-analytical surfaces are computed for various Dirichlet boundary conditions so as to demonstrate the versatility and robustness of the proposed method.

\section{Continuous equations}
\label{sec:continuous}
\subsection{Modeling of the Miura fold}
The Miura fold is based on the reference cell sketched in Figure \ref{fig:Miura cell}, in which all edges have unit length.
\begin{figure}
\centering
\begin{tikzpicture} [scale=1.5]
\pgfmathsetmacro{\r}{1}
\coordinate (a) at (0,0);
\coordinate (b) at (\r,0);
\coordinate (c) at (2*\r,0);
\coordinate (d) at \r*(0.5,{sqrt(2)/2});
\coordinate (e) at ($(d)+(\r,0)$);
\coordinate (f) at ($(e)+(\r,0)$);
\coordinate (g) at \r*(0.5,{-sqrt(2)/2});
\coordinate (h) at ($(g)+(\r,0)$);
\coordinate (i) at ($(h)+(\r,0)$);

\draw[-] (a) -- (b);
\draw[-] (c) -- (b);
\draw[-] (a) -- (d);
\draw[-] (e) -- (d);
\draw[-] (e) -- (f);
\draw[-] (c) -- (f);
\draw[-] (b) -- (e);
\draw[-] (c) -- (i);
\draw[-] (a) -- (g);
\draw[-] (g) -- (h);
\draw[-] (h) -- (i);
\draw[-] (b) -- (h);

\draw[-,dashed] (d) -- (b);
\draw[-,dashed] (e) -- (c);
\draw[-,dashed] (c) -- (h);
\draw[-,dashed] (g) -- (b);
\end{tikzpicture}
\caption{Miura reference cell.}
\label{fig:Miura cell}
\end{figure}
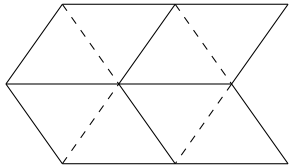
The reference cell is made of four parallelograms and can be folded along the full lines in Figure \ref{fig:Miura cell}.
Following \cite{wei2013geometric,schenk2013geometry}, we consider that the cells can also bend along the dashed lines.
However, the cells cannot stretch.
A Miura tessellation is based on the continuous juxtaposition of reference cells, dilated by a factor $r > 0$.
In the spirit of homogenization, \cite{nassar2017curvature,lebee2018fitting} have proposed a procedure to compute a surface that is the limit when $r \to 0$ of a Miura tessellation, see Figure \ref{fig:homogenization}.
This procedure leads to a constrained PDE described in \cite{lebee2018fitting}.
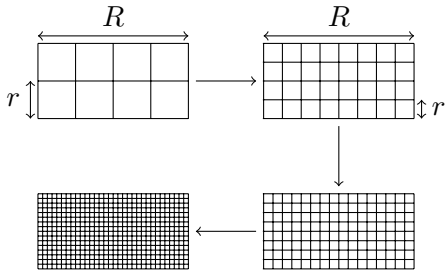
\begin{figure}
\centering
\begin{tikzpicture}
\draw [to-to] (0,1.1) -- (2,1.1);
\node[above] at (1, 1.1) {$R$};
\node[left] at (-0.1, 0.25) {$r$};
\draw [to-to] (-0.1,0) -- (-0.1,0.5);
\draw[step=0.5,black,thin] (0,0) grid (2,1);

\draw [to-to] (3,1.1) -- (5,1.1);
\node[above] at (4, 1.1) {$R$};
\node[right] at (5.1, 0.125) {$r$};
\draw [to-to] (5.1,0) -- (5.1,0.25);
\draw[step=0.25,black,thin] (2.99,0) grid (5,1);

\draw[step=0.125,black,thin] (2.99,-2) grid (5,-1);

\draw[step=0.0625,black,thin]  (0,-2) grid (2,-1);

\draw [-to](2.1,0.5) -- (2.9,0.5); 
\draw [-to](4,-0.1) -- (4,-0.9); 
\draw [-to] (2.9,-1.5) -- (2.1,-1.5); 
\end{tikzpicture}
\caption{Homogenization of a Miura tessellation}
\label{fig:homogenization}
\end{figure}

\subsection{Strong form equations}
Let $\Omega \subset \mathbb{R}^2$ be a bounded convex polygon that can be perfectly fitted by triangular meshes.
Note that, due to the convexity hypothesis, the boundary $\partial \Omega$ is Lipschitz \cite{grisvard2011elliptic} and verifies an exterior sphere condition \cite{adams2003sobolev}.
Let $\varphi : \Omega \subset \mathbb{R}^2 \rightarrow \mathbb{R}^3$ be a parametrization of the homogenized surface constructed from a Miura tessellation.
The coordinates of $\varphi$ are written as $\varphi^i$, for $i \in \{1,2,3\}$.
As proved in \cite{lebee2018fitting}, $\varphi$ is a solution of the following strong form equation:
\begin{equation}
\label{eq:min surface eq}
p(\varphi_x) \varphi_{xx} + q(\varphi_y) \varphi_{yy} = 0 \in \mathbb{R}^3,
\end{equation}
where 
\[p(\varphi_x) = \frac{1}{4 - |\varphi_{x}|^2}, \quad q(\varphi_y) = \frac{1}{|\varphi_{y}|^2}, \]
and the subscripts $x$ and $y$ stand respectively for $\partial_x$ and $\partial_y$.
It is also proved in \cite{lebee2018fitting} that solutions to \eqref{eq:min surface eq} should verify $|\varphi_y|^2 > 1$.
Note that Equation \ref{eq:min surface eq} is a simplification of the actual equation derived in \cite{lebee2018fitting}.
The equation in \cite{lebee2018fitting} also comprises equality constraints which constitute a challenge in themselves.
In this paper, we restrict ourselves to the simpler problem of studying \eqref{eq:min surface eq} with the constraint $|\varphi_y|^2 > 1$ and Dirichlet boundary conditions on all of $\partial \Omega$.

\begin{remark}
Note that, because \eqref{eq:min surface eq} is derived from zero energy deformation modes, it is not variational in the sense that it does not derive from an energy that could be interpreted as the elastic energy of the system.
This has implications in the proof of existence of solutions as variational techniques cannot be used.
\end{remark}

We impose strongly the Dirichlet boundary conditions $\varphi = \varphi_D$ on $\partial \Omega$ where $\varphi_D \in \left( \mathcal{C}^0(\partial \Omega)\right)^3$.
$\varphi_D$ is actually assumed to be more regular, as it should verify a bounded slope condition.

\begin{definition}[Bounded slope condition]
\label{def:bounded slope}
Let $u$ be a function defined on $\partial \Omega$ and 
\[ \Gamma := \{(x,y,z) \in \partial \Omega \times \mathbb{R} \ | \ z = u(x,y) \}. \]
$u$ verifies a bounded slope condition over $\partial \Omega$ with constant $K > 0$, if for every point $P \in \Gamma$, there exists two planes $P_\pm$ in $\mathbb{R}^3$, $z = \pi^\pm_P(x,y)$ passing through $P$ such that
\[ \left\{
\begin{aligned}
&\pi^-(x,y) \le u(x,y) \le \pi^+(x,y), \quad \forall (x,y) \in \partial \Omega, \\
& |\nabla \pi_P^\pm| \le K, \quad \forall P \in \Gamma.
\end{aligned}
\right. \]
Note that the last condition consists in stating that the slopes of these planes are uniformly bounded, independently of $P$, by the constant $K$.
\end{definition}

For more details, see \cite[p.~309]{gilbarg2015elliptic}.
We assume that each component $\varphi_D^i$, where $i \in \{1,2,3\}$, verifies Definition \ref{def:bounded slope}, with constants $K_i > 0$.

\begin{hyp}
\label{hyp:upper bound}
Let $K := (K_1, K_2, K_3) \in \mathbb{R}^3$, be the constants from the bounded slope condition.
It is assumed thereafter that $\partial \Omega$ and $\varphi_D$ are such that,
\begin{equation}
|K|^2 < 4.
\end{equation}
\end{hyp}
Let us write in the following $\eta := 4 - |K|^2$.

\subsection{Continuous setting}
We introduce the Hilbert space $V:=\left(H^2(\Omega)\right)^3$.
We consider the convex subset $V_D := \{\varphi \in V \ | \ \varphi=\varphi_D \text{ on } \partial \Omega\}$
as our solution space and the corresponding homogeneous space is $V_0 := \{\varphi \in V \ | \ \varphi=0  \text{ on } \partial \Omega\}$.
$V$ is equipped with the usual $\left(H^2(\Omega)\right)^3$ Sobolev norm.
Note that due to Rellich--Kondrachov theorem \cite[Theorem 9.16]{brezis}, $V \subset \left(\mathcal{C}^0(\bar{\Omega}) \right)^3$.
Let $A:V \mapsto \mathbb{R}^3$ be the operator defined for $\varphi \in V$ and $\psi \in V$ as
\begin{equation}
\label{eq:original operator}
A(\varphi) \psi := p(\varphi_x) \psi_{xx} + q(\varphi_y) \psi_{yy} \in \mathbb{R}^3.
\end{equation}
As the operator $A(\varphi)$ has no cross derivative terms, the maximum principle can be applied to each individual component.
Note that $A(\varphi)$ is not uniformly elliptic as one might have $q(\varphi_y) \to +\infty$ and is not even elliptic as one might have $p(\varphi_x) < 0$.
Therefore, let us thus define for $\varphi \in V$,
\[\bar{p}(\varphi_x) := \left\{ \begin{array}{cc} \frac1\eta & \text{if } |\varphi_{x}|^2 \ge |K|^2 \\
p(\varphi_x) & \text{otherwise} \\
 \end{array} \right.,
 \quad \bar{q}(\varphi_y) := \left\{ \begin{array}{cc} 1 & \text{if } |\varphi_{y}| \le 1 \\
 \frac1{|K|^2} & \text{if } |\varphi_{y}|^2 \ge |K|^2 \\
q(\varphi_y) & \text{otherwise} \\
 \end{array} \right. \]
$\bar{p}$ and $\bar{q}$ are Lipschitz continous with respect to their arguments and bounded.
We then define for $\varphi \in V$ and $\psi \in V$, the following operator
\begin{equation}
\label{eq:def Ah}
\bar{A}(\varphi)\psi := \bar{p}(\varphi_x) \psi_{xx} + \bar{q}(\varphi_y) \psi_{yy}.
\end{equation}
For $\varphi \in V$, $\bar{A}(\varphi)$ is uniformly elliptic.
We decide to work with the uniformly elliptic operator $\bar{A}(\varphi)$, instead of $A(\varphi)$ because it is more easily amenable to numerical approximation.
However, when $|\varphi_y| < 1$, \eqref{eq:original operator} and \eqref{eq:def Ah} do not coincide.
This is explored further, numerically, in Section \ref{sec:tests}.
 
Solving \eqref{eq:min surface eq} consists in finding $\varphi \in V_D$ such that
\begin{equation}
\label{eq:nonlinear operator eq}
\bar{A}(\varphi)\varphi = 0 \in \mathbb{R}^3.
\end{equation}
The main result of this section is Theorem \ref{th:existence} below.
The proof will follow a similar path to the proof of \cite[Theorem 12.5]{gilbarg2015elliptic}.
It consists in getting regularity from the linear equation obtained by freezing the coefficients of $\bar{A}(\varphi)$ and then using a fixed point argument.

\subsection{Existence}
Let $\varphi \in V_D$.
We first focus on solving a linear problem related to the nonlinear problem \eqref{eq:nonlinear operator eq}.
\begin{proposition}
\label{th:regularity}
There exists a unique $\psi \in V_{D}$ such that
\begin{equation}
\label{eq:linearized continuous}
\bar{A}(\varphi)\psi = 0.
\end{equation}
Moreover, there exists $\alpha \in (0,1)$ such that $\psi \in \left(\mathcal{C}^{2,\alpha}(\Omega)\right)^3$.
The solution $\psi$ also verifies the following gradient estimate,
\begin{equation}
\label{eq:gradient estimate}
\sup_\Omega |\nabla \psi^i| \le K_i, \quad \forall i \in \{1,2,3\},
\end{equation}
where $K_i > 0$ are the constants from the bounded slope condition.
\end{proposition}

\begin{proof}
$\Omega$ being a bounded convex domain, it satisfies an exterior sphere cone condition at every boundary point.
Also, $\bar{p}(\varphi)$ and $\bar{q}(\varphi)$ are Hölder continuous and $\varphi_D \in \left(\mathcal{C}^0(\partial \Omega)\right)^3$.
We can thus apply the classical result, Theorem 6.13 of \cite{gilbarg2015elliptic} for strongly elliptic linear equations to obtain the existence of $\psi \in V_D$, solution of \eqref{eq:linearized continuous}.
The $\mathcal{C}^{2,\alpha}$ regularity follows from the same theorem.
The fact that $\bar{A}(\varphi)$ is diagonal is fundamental as the cited result is proved using the maximum principle, which is not true in general for systems.
Lemma 12.6 of \cite{gilbarg2015elliptic} can finally be applied to obtain \eqref{eq:gradient estimate} because $\partial \Omega$ and $\varphi_D$ are assumed to verify a bounded slope condition with constants $(K_i)_i$.
\end{proof}

Let us now focus on the fixed point argument.
Let $T: V_D \ni \varphi \mapsto \psi(\varphi) \in V_D$ be the map that, given a $\varphi \in V_D$, associates the solution to \eqref{eq:linearized continuous}.

\begin{theorem}[Existence of a regular solution]
\label{th:existence}
There exists a solution $\varphi \in V_D$ of \eqref{eq:nonlinear operator eq}.
$\varphi$ has the following extra regularity: there exists $\alpha \in (0,1)$, $\varphi \in \left(\mathcal{C}^{2,\alpha}(\Omega)\right)^3$.
\end{theorem}

\begin{proof}
We use the Schauder fixed point theorem, see Corollary 11.2 of \cite{gilbarg2015elliptic}. The proof consists of three steps.
\paragraph{Stability under $T$}
The aim here is to construct a
set $B \subset \left(\mathcal{C}^{1,\alpha}(\Omega)\right)^3$, stable under $T$: $T(B) \subset B$.
First, let us notice that using the maximum principle, see Theorem 3.7 of \cite{gilbarg2015elliptic}, one has
\begin{equation}
\label{eq:C0 bound}
\Vert \psi \Vert_{\mathcal{C}^0(\Omega)} = \max_\Omega |\psi| \le \max_{\partial \Omega} |\varphi_D| = \Vert \varphi_D \Vert_{\mathcal{C}^0(\partial \Omega)}.
\end{equation}
We have a $\mathcal{C}^1$ bound on $\psi$ because of \eqref{eq:gradient estimate}.
We now give a Hölder estimate of $\nabla \psi$.
Let $|\cdot|_\alpha^*$ for $\alpha \in (0,1)$ be a semi-norm such that
\[ |\nabla \psi|_\alpha^* := \sup_{\substack{z,z' \in \Omega \\ z \neq z'}} \min\left(\mathrm{dist}(z,\partial \Omega), \mathrm{dist}(z',\partial \Omega)\right)^{1+\alpha} \frac{|\nabla \psi(z) - \nabla \psi(z')|}{|z - z'|^\alpha}. \]
Using \cite[Theorem 12.4]{gilbarg2015elliptic}, one has the following bound:
\begin{equation}
\label{eq:holder bound}
|\nabla \psi|_\alpha^* \le C \Vert \psi \Vert_{\mathcal{C}^0(\Omega)} \le C \Vert \varphi_D \Vert_{\mathcal{C}^0(\partial \Omega)},
\end{equation}
where $C> 0$ is independent of $\varphi$.
We define,
\[ B= \left\{ \tilde{\psi} \in \left(\mathcal{C}^{1,\alpha}(\Omega)\right)^3 ; \Vert \tilde{\psi} \Vert_{\mathcal{C}^0} \le \Vert \varphi_D \Vert_{\mathcal{C}^0(\partial \Omega)}, \quad \sup_\Omega |\nabla \tilde{\psi}|^2 \le |K|^2,
 \quad |\nabla \tilde{\psi}|_\alpha^* \le C \Vert \varphi_D \Vert_{\mathcal{C}^0(\partial \Omega)} \right\}, \]
which is a closed convex subset of the Banach space $\left(\mathcal{C}^1(\Omega)\right)^3$ associated to the semi-norm:
\[ |\nabla \tilde{\psi}|_1^* := \sup_{z \in \Omega} \mathrm{dist}(z,\partial \Omega) |\nabla \tilde{\psi}(z)|. \]
Using \eqref{eq:C0 bound}, \eqref{eq:gradient estimate} and \eqref{eq:holder bound}, one notices that $T(B) \subset B$.
Let $\Vert \cdot \Vert^*_1$ denote the norm associated with the semi-norm $|\cdot|^*_1$ applied to the gradient.

\paragraph{Precompactness of $T(B)$}
The proof of this result is based on \cite[Lemma 6.33]{gilbarg2015elliptic} which is similar to the Ascoli--Arzel\`a theorem, see \cite[Theorem 4.25]{brezis}.
By definition, the functions of $B$ are equicontinuous at every point in $\Omega$.
Let us now prove a similar result on $\partial \Omega$.
Let $z \in \Omega$ and $z_0 \in \partial \Omega$.
Using Remark 3, p.~105 of \cite{gilbarg2015elliptic}, there exists a barrier function $w$, independent of $\varphi$, such that for $\epsilon > 0$,
\[ |\psi(z) - \varphi_D(z_0)| \le \epsilon + k_\epsilon w(z), \]
where $k_\epsilon$ is independent of $\varphi \in V_D$.
Using $w(z) \to 0$, when $z \to z_0$, one proves the equicontinuity of $T(B)$ at $z_0 \in \partial \Omega$.
Thus the functions of $T(B)$ are equicontinuous over $\bar{\Omega}$ and since $T(B)$ is a bounded equicontinuous subset of $\left(\mathcal{C}^{1,\alpha}(\Omega)\right)^3$, using \cite[Lemma 6.33]{gilbarg2015elliptic}, $T(B)$ is precompact in $\left(\mathcal{C}^1(\Omega)\right)^3$.

\paragraph{Continuity of $T$}
We prove that $T$ is continuous over $\left(\mathcal{C}^{1}(\Omega)\right)^3$ for the norm $\Vert \cdot \Vert^*_1$.
Let $\left(\varphi_n\right)_n \in B^{\mathbb{N}}$, such that $\varphi_n \mathop{\longrightarrow} \limits_{n \to +\infty} \varphi \in B$ for the norm $\Vert \cdot \Vert^*_1$.
Let $\psi := T \varphi$ and $\psi_n := T \varphi_n$, for $n \in \mathbb{N}$.
We want to prove that $\psi_n \mathop{\longrightarrow} \limits_{n \to +\infty} \psi$ for $\Vert \cdot \Vert^*_1$.
An immediate consequence of \eqref{eq:gradient estimate} is the equicontinuity of $\left(\nabla \psi_n \right)_n$, on compact subdomains of $\Omega$.
Using \cite[Corollary 6.3]{gilbarg2015elliptic}, one has in particular,
\[ d^2 \Vert \nabla^2 \psi_n \Vert_{\mathcal{C}^0(\Omega')} + d^{2+\alpha} | \nabla^2 \psi_n |^*_{\alpha, \Omega'} \le C \Vert \psi_n \Vert_{\mathcal{C}^0(\Omega)} \le C \Vert \varphi_D \Vert_{\mathcal{C}^0(\partial \Omega)}, \]
where $d \le \mathrm{dist}(\Omega', \partial \Omega)$ and $\Omega' \subset \mathring{\Omega}$.
The constant $C>0$ above does not depend on $\varphi_n$ because it depends on the $\mathcal{C}^\alpha$ norm of $\bar{p}(\varphi_n)$ and $\bar{q}(\varphi_n)$ with $(\varphi_n)_n$ bounded in the $\mathcal{C}^\alpha$ norm and $\bar{p}$ and $\bar{q}$ are Lipschitz.
Therefore, $\left(\nabla^2 \psi_n \right)_n$ is equicontinuous on compact subdomains of $\Omega$.
Therefore, a subsequence $(\psi_{u(n)})_n$, converges uniformly for all $\Omega' \subset \mathring{\Omega}$, in $\mathcal{C}^{2,\alpha}(\Omega')$ towards a function $\hat{\psi} \in \mathcal{C}^{2,\alpha}(\Omega)$ verifying $A(\varphi) \hat{\psi} = 0$ in $\Omega$.
Let us now look at the boundary conditions verified by $\hat{\psi}$.
Let $z \in \Omega$ and $z_0 \in \partial \Omega$.
Using again an argument similar to \cite[Remark 3, p.~105]{gilbarg2015elliptic}, one has
\[ |\hat{\psi}_{u(n)}(z) - \varphi_D(z_0)| \le \epsilon + k_\epsilon w(z), \]
where $w$ is independent of $n$,
and thus $\hat{\psi}(z) \to \varphi_D(z_0)$, when $n \to \infty$ and $z \to z_0$.
Thus, by uniqueness of the solutions of \eqref{eq:linearized continuous}, one has $\hat{\psi} = \psi.$
Therefore $T \varphi_{u(n)} \to \psi = T \varphi$ in $\overline{\Omega}$, when $n \to \infty$.
As $T(B)$ is precompact in $\left( \mathcal{C}^1(\Omega) \right)^3$, $T \varphi_{u(n)} \to \psi$, for $\Vert \cdot \Vert_1^*$.
The expected result has been proved only for a subsequence.
The reasoning above applies to any subsequence of $(T \psi_n)_n$ and thus $(T \psi_n)_n$ has $\psi$ as its only accumulation point.
As the sequence $(T \psi_n)_n$ is bounded in $\left( \mathcal{C}^1(\Omega) \right)^3$ for $\Vert \cdot \Vert_1^*$ and has a unique accumulation point, the entire sequence $(T \psi_n)_n$ converges to $\psi$ for $\Vert \cdot \Vert_1^*$.
Therefore, $T$ is continuous.

\paragraph{Conclusion}
Using a Schauder fixed point theorem, see \cite[Corollary 11.2]{gilbarg2015elliptic},
the map $T$ admits a fixed point $\varphi^* \in V_{D}$.
Therefore, $\varphi^*$ is a solution of \eqref{eq:nonlinear operator eq} and using Proposition \ref{th:regularity}, $\varphi^* \in \left(\mathcal{C}^{2,\alpha}(\Omega) \right)^3$, for $0 < \alpha < 1$.
\end{proof}

\subsection{Weak form equation}
Following \cite{smears}, we consider a test function $\tilde{\psi} \in V_0$ and $\varphi \in V_D$ and define the form
\begin{equation}
\label{eq:main bilinear form}
a(\varphi,\tilde{\psi}) := \int_{\Omega} \Gamma(\varphi) \bar{A}(\varphi) \varphi \cdot \Delta \tilde{\psi},
\end{equation}
where $\Delta:\mathbb{R}^3 \to \mathbb{R}^3$ is the Hodge Laplacian, which computes the Laplacian of each of the coordinates and 
\[ 0 < \Gamma(\nabla \varphi) := \frac{\bar{p}(\varphi_x) + \bar{q}(\varphi_y)}{\bar{p}(\varphi_x)^2 + \bar{q}(\varphi_x)^2} \le 4\left(1 + \frac1\eta \right). \]
As $\Gamma(\varphi) \le 4 + \frac4\eta$, $a(\varphi)$ is well defined.
Equation \eqref{eq:nonlinear operator eq} is thus reformulated into search for $\varphi \in V_D$ such that
\begin{equation}
\label{eq:weak form}
a(\varphi,\tilde{\psi}) = 0, \quad \forall \tilde{\psi} \in V_0.
\end{equation}

\begin{lemma}
Equations \eqref{eq:nonlinear operator eq} and \eqref{eq:weak form} are equivalent.
\end{lemma}

\begin{proof}
Proving that solutions of \eqref{eq:nonlinear operator eq} verify \eqref{eq:weak form} is trivial.
Let us consider $\varphi \in V_D$ solution of \eqref{eq:weak form}. Let $\tilde{\psi} \in V_0$.
Classically, there exists a unique $ \Psi \in \left(H^2(\Omega)\right)^3$ such that $\Delta \Psi = \tilde{\psi}$ in $\left(L^2(\Omega)\right)^3$ and $\Psi = 0$ on $\partial \Omega$.
Thus
\[\int_\Omega \Gamma(\varphi) \bar{A}(\varphi)\varphi \cdot \Psi = 0, \quad \forall \Psi \in \left(L^2(\Omega)\right)^3.\]
Therefore, as $\Gamma(\varphi) > 0$ a.e. in $\Omega$, $\varphi$ solves \eqref{eq:nonlinear operator eq}.
\end{proof}

To prove the uniqueness of solutions to \eqref{eq:nonlinear operator eq}, we define for $\varphi \in W_{D}$, the auxiliary bilinear form such that, for $\psi \in V_D$ and $\tilde{\psi} \in V_0$,
\begin{equation}
\label{eq:linear weak}
\mathfrak{a}(\varphi;\psi,\tilde{\psi}) := \int_{\Omega} \Gamma(\varphi) \bar{A}(\varphi)  \psi \cdot \Delta \tilde{\psi}.
\end{equation}
Using $\mathfrak{a}$, we are going to prove that, under appropriate assumptions, the map $T$ of Theorem \ref{th:existence} is contracting, which provides uniqueness of a solution of \eqref{eq:nonlinear operator eq}.
First, we need to show that $\mathfrak{a}$ is coercive.
We follow \cite{smears} and use the following lemma, first.

\begin{lemma} [Cordés condition]
There exists $\varepsilon \in (0, 1]$, for all $\varphi \in V_D$,
\[ \frac{\bar{p}(\varphi_x)^2 + \bar{q}(\varphi_y)^2}{(\bar{p}(\varphi_x) + \bar{q}(\varphi_y))^2} \le \frac1{1+\varepsilon}. \]
\end{lemma}

The proof is omitted for concision.
One can refer to \cite{smears}. 
The fact that $\bar{A}(\varphi)$ is uniformly elliptic is crucial in proving the lemma.

\begin{lemma}[Coercivity]
\label{th:coercivity}
For $\varphi \in V_D$, the bilinear form $\mathfrak{a}(\varphi)$ is coercive over $V_0 \times V_0$.
There exists $C>0$, independent of $\varphi$, for all $\Psi \in V_0$,
\[ C \Vert \Psi \Vert^2_{H^2(\Omega)} \le \mathfrak{a}(\varphi; \Psi, \Psi). \]
\end{lemma}

\begin{proof}
Following \cite{smears}, we use the Miranda--Talenti theorem which states, there exists $C > 0$, for all $\Psi \in \left(H^2(\Omega) \right)^3 \cap \left(H^1_0(\Omega) \right)^3$,
\[
\begin{aligned}
& |\Psi|_{H^2(\Omega)} \le \Vert \Delta \Psi \Vert_{L^2(\Omega)}, \\
& \Vert \Psi \Vert_{H^2(\Omega)} \le C \Vert \Delta \Psi \Vert_{L^2(\Omega)},
\end{aligned}
 \]
where $C$ is a constant depending only on the diameter of $\Omega$.
Let us now prove the coercivity of $\mathfrak{a}(\varphi)$. Using \cite[Lemma 1]{smears}, one has
\[ \begin{alignedat}{1}
\mathfrak{a}(\varphi;\Psi,\Psi) &= \Vert \Delta \Psi \Vert_{L^2} - \int_\Omega (\Delta - \Gamma(\varphi) \bar{A}(\varphi)) \Psi \cdot \Delta \Psi \\
&\ge \Vert \Delta \Psi \Vert_{L^2}^2 - \sqrt{1 - \varepsilon} |\Psi|_{H^2(\Omega)} \Vert \Delta \Psi \Vert_{L^2} \\
& \ge \frac{1 - \sqrt{1 - \varepsilon}}{C^2} \Vert \Psi \Vert_{H^2(\Omega)}^2,
\end{alignedat} \]
where $C > 0$ is the constant from the Miranda--Talenti theorem.
\end{proof}

\subsection{Uniqueness}
This section requires a little more regularity on the Dirichlet boundary condition $\varphi_D$.

\begin{hyp} [Additional regularity]
Let us assume that $\varphi_D \in H^\frac32(\partial \Omega)^3$.
\end{hyp}

\begin{lemma}
\label{th:useful}
Let $\varphi \in V_D$ and $\psi \in V_D$ verifying \eqref{eq:linear weak}.
There exists $C > 0$, independent of $\varphi$,
\[ \Vert \psi \Vert_{H^2(\Omega)} \le C \Vert \varphi_D \Vert_{H^\frac32(\Omega)}. \]
\end{lemma}

\begin{proof}
As $\varphi_D \in H^\frac32(\partial \Omega)^3$, there exists $\Phi_D \in V$ such that $\varphi_D = \Phi_D$ on $\partial \Omega$ and there exists $C>0$,
\[ \Vert \Phi_D \Vert_{H^2(\Omega)} \le C \Vert \varphi_D \Vert_{H^\frac32(\partial \Omega)}. \]
Let $\zeta := \psi - \Phi_D \in V_0$, one thus has
\[ a(\varphi; \zeta, \zeta) = -a(\varphi; \Phi_D, \zeta). \]
Using Lemma \ref{th:coercivity}, one has
\[ c \Vert \zeta \Vert_{H^2(\Omega)}^2 \le C' \Vert \Phi_D \Vert_{H^2(\Omega)} \Vert \zeta \Vert_{H^2(\Omega)}, \]
where $c > 0$ is the coercivity constant from Lemma \ref{th:coercivity} and $C' > 0$, independent of $\varphi$, comes from the fact that $\Gamma(\varphi) \bar{A}(\varphi)$ is a bounded operator.
One thus has
\[ \Vert \zeta \Vert_{H^2(\Omega)} \le \frac{C'}c \Vert \Phi_D \Vert_{H^2(\Omega)}. \]
Using the definition of $\zeta$, and the extension inequality, one gets the expected result.
\end{proof}

\begin{proposition}[Uniqueness]
Let $\delta > 0$ and $\Vert \varphi_D \Vert_{H^\frac32(\partial \Omega)} \le \delta$.
For $\delta$ small enough, Equation \eqref{eq:nonlinear operator eq} admits a unique solution.
\end{proposition}

\begin{proof}
We prove that $T$ is contracting.
Let $\varphi,\hat{\varphi} \in V_D$.
There exists $\psi,\hat{\psi} \in V_D$,
\[ \mathfrak{a}(\varphi;\psi, \tilde{\psi}) = 0 = \mathfrak{a}(\hat{\varphi};\hat{\psi}, \tilde{\psi}), \quad \forall \tilde{\psi} \in V_0. \]
As $\varphi_D \in H^\frac32(\partial \Omega)^3$, there exists $\Phi_D \in V$ such that $\varphi_D = \Phi_D$ on $\partial \Omega$. 
Let $\hat{\zeta} := \hat{\psi} - \Phi_D$ and $\zeta := \psi - \Phi_D$.
Let $\tilde{\psi} \in V_0$, one thus has
\[ \left\{ \begin{aligned}
& a(\varphi; \zeta, \tilde{\psi}) = a(\varphi; \Phi_D, \tilde{\psi}), \\
& a(\hat{\varphi}; \hat{\zeta}, \tilde{\psi}) = a(\hat{\varphi}; \Phi_D, \tilde{\psi}).
\end{aligned} \right. \]
Therefore,
\begin{align*}
\int_\Omega \left( \Gamma(\varphi) \bar{A}(\varphi) - \Gamma(\hat{\varphi}) \bar{A}(\hat{\varphi}) \right)\Phi_D \cdot \Delta \tilde{\psi} &= \mathfrak{a}(\varphi;\zeta, \tilde{\psi}) - \mathfrak{a}(\hat{\varphi};\hat{\zeta}, \tilde{\psi}), \\
& = \mathfrak{a}(\varphi;\zeta, \tilde{\psi}) - \mathfrak{a}(\varphi;\hat{\zeta}, \tilde{\psi}) + \mathfrak{a}(\varphi;\hat{\zeta}, \tilde{\psi}) - \mathfrak{a}(\hat{\varphi};\hat{\zeta}, \tilde{\psi}).
\end{align*}
Thus,
\[ \int_\Omega \left( \Gamma(\varphi) \bar{A}(\varphi) - \Gamma(\hat{\varphi}) \bar{A}(\hat{\varphi}) \right)(\Phi_D - \hat{\zeta}) \cdot \Delta \tilde{\psi} = \mathfrak{a}(\varphi;\zeta - \hat{\zeta}, \tilde{\psi}). \]

Using a Cauchy-Schwarz inequality, one has
\[ \vert \mathfrak{a}(\varphi;\zeta - \hat{\zeta}, \tilde{\psi}) \vert \le \left\Vert \left( \Gamma(\varphi) \bar{A}(\varphi) - \Gamma(\hat{\varphi}) \bar{A}(\hat{\varphi}) \right)(\Phi_D - \hat{\zeta}) \right\Vert_{L^2(\Omega)} \cdot \Vert \Delta \tilde{\psi} \Vert_{L^2(\Omega)}. \]
Because $\Gamma \bar{A}$ is Lipschitz, one has a.e. in $\Omega$,
\[ \left\vert \left( \Gamma(\varphi) \bar{A}(\varphi) - \Gamma(\hat{\varphi}) \bar{A}(\hat{\varphi}) \right)(\Phi_D - \hat{\zeta}) \right\vert \le K \vert \nabla \varphi - \nabla \hat{\varphi} \vert \left( |\Phi_{D,xx} - \hat{\zeta}_{xx}| + |\Phi_{D,yy} - \hat{\zeta}_{yy}| \right), \]
where $K > 0$ is the Lipschitz constant.
Therefore,
\[ \left\Vert \left( \Gamma(\varphi) \bar{A}(\varphi) - \Gamma(\hat{\varphi}) \bar{A}(\hat{\varphi}) \right)(\Phi_D - \hat{\zeta}) \right\Vert_{L^2(\Omega)} \le \sqrt{2} K |\varphi - \hat{\varphi}|_{W^{1,\infty}(\Omega)} \left( \Vert \Phi_D \Vert_{H^2(\Omega)} + \Vert \hat{\zeta} \Vert_{H^2(\Omega)} \right)  \]
Testing with $\tilde{\psi} = \hat{\zeta} - \zeta \in V_0$, and using Lemma \ref{th:coercivity}, with coercivity constant $c > 0$, one has
\[ c \Vert \zeta - \hat{\zeta} \Vert^2_{H^2(\Omega)} \le \sqrt{2} K \Vert \varphi - \hat{\varphi} \Vert_{W^{1,\infty}(\Omega)} \left( \Vert \Phi_D \Vert_{H^2(\Omega)} + \Vert \hat{\zeta} \Vert_{H^2(\Omega)} \right) \Vert \zeta - \hat{\zeta} \Vert_{H^2(\Omega)}. \]
Thus,
\[ c \Vert \psi-\hat{\psi} \Vert_{H^2(\Omega)} \le \sqrt2 K \Vert \varphi - \hat{\varphi} \Vert_1^* (C' + C)\delta, \]
where $C>0$ is the constant from Lemma \ref{th:useful} and $C'> 0$ is the constant from the extension $\Phi_D$.
If $\delta$ is small enough, then $T$ is contracting and \eqref{eq:nonlinear operator eq} admits a unique solution.
\end{proof}

\section{Numerical scheme}
\label{sec:discrete}
Approximate solutions to \eqref{eq:weak form} are computed using $H^2$-conformal finite elements and a Newton method.
\subsection{Discrete Setting}
Let $\left(\mathcal{T}_h \right)_h$ be a family of quasi-uniform and shape regular triangulations \cite{ern_guermond}, perfectly fitting $\Omega$.
For a cell $c \in \mathcal{T}_h$, let $h_c := \mathrm{diam}(c)$ be the diameter of c.
Then, we define $h := \max_{c \in \mathcal{T}_h} h_c$ as the mesh parameter for a given triangulation $\mathcal{T}_h$ and $\mathcal{E}_h$ as the set of its edges.
The set $\mathcal{E}_h$ is partitioned as $\mathcal{E}_h := \mathcal{E}_h^i \cup \mathcal{E}_h^b$, where for all $e \in \mathcal{E}_h^b$, $e \subset \partial \Omega$ and $\cup_{e \in \mathcal{E}^b_h} {e} = \partial \Omega$.

As \eqref{eq:weak form} is written in a subset of $\left(H^2(\Omega)\right)^3$, we resort to discretizing it using vector Bell FEM \cite{bell1969refined}, which are $H^2$-conformal.
Let
\[ V_h := \left\{ \varphi_h \in \mathcal{C}^1(\Omega)^3 \ | \ \forall c \in \mathcal{T}_h, {\varphi_h}_{|c} \in \mathbb{P}_5(c)^3, \ \forall e \in \mathcal{E}_h, \frac{\partial \varphi_h}{\partial n_e} \in \mathbb{P}_3(e)^{3 \times 2} \right\},\]
where $n_e$ is the normal to an edge $e \in \mathcal{E}_h$.
Let $V_{hD} := \{\varphi_h \in V_h | \varphi_h = \mathcal{I}_h \varphi_D \text{ on } \partial \Omega \}$, where $\mathcal{I}_h$ is the Bell interpolant \cite{brenner2008mathematical} and $V_{h0}$ is the corresponding homogeneous space.
We write the discrete problem as: search for $\varphi_h \in V_{hD}$, such that,
\begin{equation}
\label{eq:discrete equation}
a(\varphi_{h},\tilde{\psi}_{h}) = \int_\Omega \Gamma(\varphi_h ) \bar{A}(\varphi_h) \varphi_h \cdot \Delta \tilde{\psi}_{h} = 0, \quad \forall \tilde{\psi}_h \in V_{h0}.
\end{equation}
To study solutions to \eqref{eq:discrete equation}, we resort to a fixed point method.

\begin{lemma}
\label{th:unique linearized}
Given $\varphi_h \in V_{hD}$, the equation search for $\psi_h \in V_{hD}$ such that
\begin{equation}
\label{eq:linearized discrete}
\mathfrak{a}(\varphi_h;\psi_h,\tilde{\psi}_h) = 0, \quad \forall \tilde{\psi}_h \in V_{h0},
\end{equation}
admits a unique solution.
\end{lemma}

\begin{proof}
Let $\varphi_h \in V_{hD}$.
As $V_{h0} \subset V_0$ and $\mathfrak{a}(\varphi_h)$ is coercive over $V_0 \times V_0$, as proved in Lemma \ref{th:coercivity}, then it is coercive over
$V_{h0} \times V_{h0}$.
Let $\Phi_D \in V_D$ and $\Psi_h := \psi_h - \mathcal{I}_h \Phi_D$.
We are now interested in searching for $\Psi_h \in V_{h0}$,
\[ \mathfrak{a}(\varphi_h;\Psi_h,\tilde{\psi}_h) = -\mathfrak{a}(\varphi_h;\mathcal{I}_h \Phi_D,\tilde{\psi}_h), \quad \forall \tilde{\psi}_h \in V_{h0}. \]
This equation has a unique solution as $\mathfrak{a}(\varphi_h)$ is coercive.
\end{proof}

\subsection{Discrete solution}
\label{sec:fixed point}
\begin{proposition}
\label{th:discrete fixed point}
The map $T_h: V_{hD} \ni \varphi_h \mapsto \psi_h(\varphi_h) \in V_{hD}$, where $\psi_h$ is the unique solution of \eqref{eq:linearized discrete}, admits a fixed point $\varphi_h^* \in V_{hD}$ which verifies
\begin{equation*}
a(\varphi^*_h, \tilde{\psi}_h)  = \mathfrak{a}(\varphi^*_h; \varphi^*_h, \tilde{\psi}_h) = 0, \quad \forall \tilde{\psi}_h \in V_{h0},
\end{equation*}
and there exists $C > 0$, independent of $h$,
\begin{equation}
\label{eq:discrete bound}
\Vert \varphi_h^* \Vert_{H^2(\Omega)} \le  C.
\end{equation}
Moreover, letting $\delta > 0$, $\Vert \varphi_D \Vert_{H^\frac32(\partial \Omega)} \le \delta$, for $\delta$ small enough, \eqref{eq:discrete equation} admits a unique solution.
\end{proposition}

\begin{proof}

Let us first find a stable convex domain $B \subset \left( H^{2}(\Omega) \right)^3$.
Let $\Phi_D \in V$ such that $\varphi_D = \Phi_D$ on $\partial \Omega$ and $\psi_h := \Psi_h + \mathcal{I}_h \Phi_D$, where $\Psi_h$ is the solution of 
\[ \mathfrak{a}(\varphi_h;\Psi_h,\tilde{\psi}_h) = -\mathfrak{a}(\varphi_h;\mathcal{I}_h \Phi_D,\tilde{\psi}_h), \quad \forall \tilde{\psi}_h \in V_{h0}. \]
Therefore, using the coercivity of $\mathfrak{a}(\varphi_h)$,
\begin{align*}
\Vert \Psi_h \Vert^2_{H^2(\Omega)} \le \frac{C^2}{1 - \sqrt{1 - \varepsilon}} \mathfrak{a}(\varphi_h;\Psi_h,\Psi_h) &= \frac{C^2}{1 - \sqrt{1 - \varepsilon}} |\mathfrak{a}(\varphi_h;\mathcal{I}_h \Phi_D,\Psi_h)| \\
&\le \frac{C^2}{1 - \sqrt{1 - \varepsilon}} C' \Vert \Phi_D \Vert_{H^2(\Omega)} \Vert \Psi_h \Vert_{H^2(\Omega)},
\end{align*}
where $C'>0$ is a constant independent of $\varphi_h$
Thus
\[\mathfrak{a}(\varphi_h;\psi_h,\tilde{\psi}_h) = 0, \quad \tilde{\psi}_h \in V_{h0}. \]
Therefore, one has
\[\Vert \psi_h \Vert_{H^2(\Omega)} \le \Vert \mathcal{I}_h \Phi_D \Vert_{H^2(\Omega)} + \Vert \Psi_h \Vert_{H^2(\Omega)} \\
\le C''\Vert \varphi_D \Vert_{H^\frac32(\partial \Omega)}, \]
where $C'' > 0$ is a constant independent of $h$.
Let $B$ be the ball of $V_h \subset \left(H^2(\Omega) \right)^3$ centred in $0$ with radius $C''\Vert \varphi_D \Vert_{H^\frac32(\partial \Omega)}$.
Thus $T_h(B) \subset B$.

Let us now prove the continuity of $T_h$.
We follow \cite{brenner2008mathematical} and actually prove that $T_h$ is Lipschitz.
Let $\varphi_h,\hat{\varphi}_h \in V_{hD}$ and $\tilde{\psi}_h \in V_{h0}$, one thus has
\[ \int_\Omega \left(\Gamma(\varphi_h) \bar{A}(\varphi_h) - \Gamma(\hat{\varphi}_h) \bar{A}(\hat{\varphi}_h) \right) T_h \varphi_h \cdot \Delta \tilde{\psi}_h = -\mathfrak{a}(\hat{\varphi}_h; T_h \varphi_h- T_h \hat{\varphi}_h,\tilde{\psi}_h). \]
Because $\Gamma \bar{A}$ is $K$-Lipschitz, one has
\begin{align*}
|\mathfrak{a}(\hat{\varphi}_h;T_h \varphi_h-T_h \hat{\varphi}_h,\tilde{\psi}_h)| &\le K \Vert \vert \nabla \varphi_h - \nabla \hat{\varphi}_h \vert \vert \Delta T_h \varphi_h \vert \Vert_{L^2(\Omega)} \Vert \Delta \tilde{\psi}_h \Vert_{L^2(\Omega)}, \\
& \le K \Vert \nabla \varphi_h - \nabla \hat{\varphi}_h \Vert_{L^p(\Omega)} \Vert \Delta T_h \varphi_h \Vert_{L^q(\Omega)} C \Vert \tilde{\psi}_h \Vert_{H^2(\Omega)}, \\
& \le K \Vert \nabla \varphi_h - \nabla \hat{\varphi}_h \Vert_{L^p(\Omega)} \Vert \Delta T_h \varphi_h \Vert_{L^2(\Omega)} C \Vert \tilde{\psi}_h \Vert_{H^2(\Omega)}, \\
& \le K C' \Vert \varphi_h - \hat{\varphi}_h \Vert_{H^2(\Omega)} C \Vert T_h \varphi_h \Vert_{H^2(\Omega)} C \Vert \tilde{\psi}_h \Vert_{H^2(\Omega)}, \\
\end{align*}
where $C > 0$ is the Mirand--Talenti constant and $C'>0$ is the constant of the Sobolev  $H^2(\Omega)^3 \subset W^{1,r}(\Omega)^3$, are independent of $h$ and the Hölder inequality was used with $\frac1{r} + \frac1{s} = \frac12$, $r>2$ and $s < 2$.
Finally,
\[ |\mathfrak{a}(\hat{\varphi}_h;T_h \varphi_h-T_h \hat{\varphi}_h,\tilde{\psi}_h)| \le C'' \Vert \varphi_D \Vert_{H^\frac32(\partial \Omega)} \Vert \varphi_h - \hat{\varphi}_h \Vert_{H^2(\Omega)} \Vert \tilde{\psi}_h \Vert_{H^2(\Omega)},\]
where $C''> 0$ is a constant independent of $h$.
Using $\tilde{\psi}_h := T_h \varphi_h - T_h \hat{\varphi}_h$ and the coercivity of $\mathfrak{a}(\hat{\varphi}_h)$, one has
\[ \Vert T_h \varphi_h - T_h \hat{\varphi}_h \Vert_{H^2} \le C'' \delta \Vert \varphi_h - \hat{\varphi}_h\Vert_{H^2(\Omega)}.\]
As a consequence of the Brouwer fixed point theorem \cite{brezis}, there exists $\varphi_h^* \in V_{hD}$ solution to \eqref{eq:discrete equation}.
Also, for $\delta$ small enough, the fixed point is unique because $T_h$ is contracting.

\end{proof}

\subsection{Convergence rate}
To be able to give a convergence rate, we make the following stronger regularity assumption on the boundary condition.
\begin{hyp}
Let us assume that $\varphi_D \in \left( H^\frac52(\Omega) \right)^3$.
\end{hyp}
Before giving the convergence rate, we prove the following lemma.

\begin{lemma}
\label{th:extra regularity}
The solution $\varphi \in V_D$ of \eqref{eq:nonlinear operator eq} is such that $\varphi \in H^3(\Omega)^3$ and there exists a constant $C > 0$, 
\[ \Vert \varphi \Vert_{H^3(\Omega)} \le C \Vert \varphi_D \Vert_{H^\frac52(\partial \Omega)}. \]
\end{lemma}

\begin{proof}
Let $\bar{p} := \bar{p}(\varphi)$, $\bar{q} := \bar{q}(\varphi)$, $\psi := \varphi_x$ and $\hat{\psi} := \varphi_y$.
We derive in the sense of distributions \eqref{eq:nonlinear operator eq} with respect to $x$ and $y$ and get
\[ \left\{ \begin{aligned} 
& \bar{p} \psi_{xx} + \bar{q} \psi_{yy} = - \bar{p}_x \varphi_{xx} - \bar{q}_x \varphi_{yy} =:f \in L^2(\Omega)^3,\\
& \bar{p} \hat{\psi}_{xx} + \bar{q} \hat{\psi}_{yy} = - \bar{p}_y \varphi_{xx} - \bar{q}_y \varphi_{yy} =:g \in L^2(\Omega)^3. \\
\end{aligned} \right. \]
As $\varphi_D \in H^\frac52(\partial \Omega)^3$, there exists $\Phi_D \in H^3(\Omega)^3$, $\Phi_D = \varphi_D$ on $\partial \Omega$.
Let $\hat{f} := -\bar{p}\Phi_{D,xxx} - \bar{q}\Phi_{D,xyy} \in L^2(\Omega)^3$, and $\Psi := \psi - \Phi_{D,x}$
Then one has
\[ \bar{p}\Psi_{xx} + \bar{q}\Psi_{yy} = f + \hat{f}. \]
Let $\Gamma_1 > 0$ a.e. in $\Omega$, $\Gamma_1 \ge \Gamma(\varphi)$.
Using Lemma \ref{th:coercivity}, one has
\begin{align*}
c \Vert \Psi \Vert_{H^2(\Omega)}^2 & \le \Gamma_1 (\Vert f \Vert_{L^2(\Omega)} + \Vert \hat{f} \Vert_{L^2(\Omega)} ) \Vert \Delta \Psi \Vert_{L^2(\Omega)}, \\
& \le \Gamma_1 (K_1 \Vert \varphi \Vert_{H^2(\Omega)} + K_2 \Vert \Phi_D \Vert_{H^3(\Omega)} ) C \Vert \Psi \Vert_{H^2(\Omega)},
\end{align*}
where $c > 0$ is the coercivity constant of $\mathfrak{a}(\varphi)$, $C> 0$ is the constant from the Miranda--Talenti lemma, and $K_1,K_2>0$ are constants independent of $\varphi$ and linked to the fact that $\bar{p},\bar{q}$ are Lipschitz.
Finally,
\[ c \Vert \Psi \Vert_{H^2(\Omega)} \le \Gamma_1 (K_1 C' + K_2 C'' ) C \Vert \varphi_D \Vert_{H^\frac52(\partial \Omega)}, \]
where $C'>0$ is the constant from Lemma \ref{th:useful}, and $C'' > 0$ is the constant from the extension $\Phi_D$.
Thus, one has $\Psi \in \left( H^2(\Omega) \right)^3$.
Using a similar treatment on the second equation, one has $\varphi \in \left( H^3(\Omega) \right)^3$ and the announced inequality.
\end{proof}

\begin{theorem}
\label{th:convergence rate}
Let $\delta > 0$, $\Vert \varphi_D \Vert_{H^\frac52(\partial \Omega)} \le \delta$.
For $\delta > 0$ small enough, the sequence $(\varphi_h)_h \in V_{hD}^\mathbb{N}$ of solutions of \eqref{eq:weak form} converges towards $\varphi \in V_D$, solution of \eqref{eq:nonlinear operator eq}, with the following convergence estimate,
\begin{equation}
\label{eq:convergence rate}
\Vert \varphi - \varphi_h \Vert_{H^{2}(\Omega)} \le C(\varphi) h,
\end{equation}
where $C(\varphi) >  0$ is a constant depending on $\varphi$.
\end{theorem}

\begin{proof}
One has
\[ \mathfrak{a}(\varphi_h;\varphi - \varphi_h, \varphi - \varphi_h) = \mathfrak{a}(\varphi_h;\varphi - \varphi_h, \varphi - \mathcal{I}_h \varphi) + \mathfrak{a}(\varphi_h;\varphi - \varphi_h, \mathcal{I}_h \varphi - \varphi_h) \]
One also has
\begin{align*}
\mathfrak{a}(\varphi_h;\varphi - \varphi_h, \mathcal{I}_h \varphi - \varphi_h) &= \mathfrak{a}(\varphi_h;\varphi, \mathcal{I}_h \varphi - \varphi_h) - 0, \\
& = \mathfrak{a}(\varphi_h;\varphi, \mathcal{I}_h \varphi - \varphi_h) - \mathfrak{a}(\varphi,\varphi, \mathcal{I}_h \varphi - \varphi_h), \\
& = \int_\Omega \left( \Gamma(\varphi_h) \bar{A}(\varphi_h) - \Gamma(\varphi) \bar{A}(\varphi) \right)\varphi \cdot \Delta (\mathcal{I}_h \varphi - \varphi_h).
\end{align*}
$\Gamma \bar{A}$ is a bounded operator.
Thus, let $M > 0$ such that, for all $(\psi,\zeta) \in V^2$, $ \Gamma(\psi) \bar{A}(\psi)\zeta \le M$ a.e. in $\Omega$.
Therefore,
\begin{multline*}
\mathfrak{a}(\varphi_h;\varphi - \varphi_h, \varphi - \varphi_h) \le M \Vert \varphi - \varphi_h \Vert_{H^2(\Omega)} \Vert \Delta (\varphi - \mathcal{I}_h \varphi) \Vert_{L^2(\Omega)} \\ + K \Vert |\nabla \varphi - \nabla \varphi_h| |\Delta \varphi| \Vert_{L^2(\Omega)} \Vert \Delta (\mathcal{I}_h \varphi - \varphi_h) \Vert_{L^2(\Omega)},
\end{multline*}
where $\Gamma \bar{A}$ is K-Lipschitz.
Let $r,s>0$ such that $\frac1{r} + \frac1{s} = \frac12$, and $r > 2$.
Let $C > 0$ be the constant from the Miranda--Talenti estimate.
Using the Hölder inequality, one has
\begin{align*}
& \mathfrak{a}(\varphi_h;\varphi - \varphi_h, \varphi - \varphi_h) \\ 
&\le M C \Vert \varphi - \varphi_h \Vert_{H^2(\Omega)} \Vert \varphi - \mathcal{I}_h \varphi \Vert_{H^2(\Omega)} + KC \Vert |\nabla \varphi - \nabla \varphi_h| |\Delta \varphi| \Vert_{L^2(\Omega)} \Vert \mathcal{I}_h \varphi - \varphi_h \Vert_{H^2(\Omega)}, \\
& \le M C \Vert \varphi - \varphi_h \Vert_{H^2(\Omega)} \Vert \varphi - \mathcal{I}_h \varphi \Vert_{H^2(\Omega)} + KC \Vert \varphi - \varphi_h \Vert_{W^{1,r}(\Omega)} \Vert \Delta \varphi \Vert_{L^{s}(\Omega)} \Vert \mathcal{I}_h \varphi - \varphi_h \Vert_{H^2(\Omega)}, \\
& \le M C \Vert \varphi - \varphi_h \Vert_{H^2(\Omega)} \Vert \varphi - \mathcal{I}_h \varphi \Vert_{H^2(\Omega)} + KC C'^2 \Vert \varphi - \varphi_h \Vert_{H^2(\Omega)} \Vert \Delta \varphi \Vert_{H^1(\Omega)} \Vert \mathcal{I}_h \varphi - \varphi_h \Vert_{H^2(\Omega)}. \\
\end{align*}
where $C' > 0$ is the largest of the two constants from the Sobolev injections $H^2(\Omega)^3 \subset W^{1,r}(\Omega)^3$ and $H^1(\Omega)^3 \subset L^{s}(\Omega)^3$.
Let $c$ the coercivity constant of $\mathfrak{a}(\varphi_h)$.
One thus has
\begin{align*}
\Vert \varphi - \varphi_h \Vert_{H^2(\Omega)} &\le \frac{MC}{c} \Vert \varphi - \mathcal{I}_h \varphi \Vert_{H^2(\Omega)} + \frac{KCC'^2}c \Vert \varphi \Vert_{H^3(\Omega)} \left( \Vert \mathcal{I}_h \varphi - \varphi \Vert_{H^2(\Omega)} + \Vert \varphi - \varphi_h \Vert_{H^2(\Omega)} \right), \\
& \le \frac{MC}{c} \Vert \varphi - \mathcal{I}_h \varphi \Vert_{H^2(\Omega)} + \frac{KCC'^2C''}c \Vert \varphi_D \Vert_{H^\frac52(\partial \Omega)} \left( \Vert \mathcal{I}_h \varphi - \varphi \Vert_{H^2(\Omega)} + \Vert \varphi - \varphi_h \Vert_{H^2(\Omega)} \right),
\end{align*}
where $C'' > 0$ is the constant from Lemma \ref{th:extra regularity}.
Let $\gamma := \frac{KCC'^2C''}c \delta < 1$, for $\delta$ small enough. 
Finally,
\[ (1 - \gamma) \Vert \varphi - \varphi_h \Vert_{H^2(\Omega)} \le \left( \frac{MC}c + \gamma \right) \Vert \mathcal{I}_h \varphi - \varphi \Vert_{H^2(\Omega)}. \]
As $\varphi \in \left(H^3(\Omega) \right)^3$, one has only the classical interpolation error \cite{brenner2008mathematical},
\[ \Vert \mathcal{I}_h \varphi  - \varphi \Vert_{H^{2}(\Omega)} \le C h \vert \varphi \vert_{H^{3}(\Omega)}. \]
Hence the announced result.
\end{proof}

\section{Numerical tests}
\label{sec:tests}
The method is implemented in the FEM software \textit{Firedrake} \cite{firedrake}.
Because it is not possible to impose strongly Dirichlet boundary conditions in \textit{Firedrake} for the Bell finite element, at the moment, we resort to a least-square penalty.
Let us define the bilinear form
\begin{equation}
b_h(\psi_h,\tilde{\psi}_h) := \bar{\eta} \sum_{e \in \mathcal{E}_h^b} h_e^{-4} \int_e \psi_h \cdot \tilde{\psi}_h,
\end{equation}
where $\bar{\eta} > 0$ is a user defined penalty coefficient and $h_e:= \mathrm{diam}(e)$ for $e \in \mathcal{E}_h$.
The corresponding right-hand side is
\begin{equation}
l_h(\tilde{\psi}_h) := \bar{\eta} \sum_{e \in \mathcal{E}_h^b} h_e^{-4} \int_e \varphi_D \cdot \tilde{\psi}_h.
\end{equation}
Several values of $\bar{\eta}$ have been tested and $\bar{\eta} = 10$ seems to work fine.
An effect would be seen if $\bar{\eta}$ was chosen too small, in which case the boundary conditions would not get imposed weakly, or too big, in which case it would affect the conditioning of the rigidity matrices.
$\bar{\eta} = 10$ is chosen for all numerical test.
To solve \eqref{eq:discrete equation}, we use a Newton method.
We thus solve at each iteration,
\[ \mathfrak{a}(\varphi_h;\psi_h,\tilde{\psi}_h) + b_h(\psi_h,\tilde{\psi}_h) = l_h(\tilde{\psi}_h), \quad \forall \tilde{\psi}_h \in V_{h0}. \]
As initial guess, we consider the solution of the Laplace equation with Dirichlet boundary conditions $\mathcal{I}_h \varphi_D$ on $\partial \Omega$:
compute $\psi_{h,0} \in V_{hD}$ solution of 
\[ \int_\Omega \nabla \psi_{h,0} \cdot \nabla \tilde{\psi}_h = 0, \quad \forall \tilde{\psi}_h \in V_{h0}. \]
The default parameters of \textit{Firedrake} are used regarding the stopping criterion for the Newton method.

Also, note that it is difficult in practice to estimate the value of the constant $|K|$ in Hypothesis \ref{hyp:upper bound}.
Instead, we define $\epsilon = 0.1$, and for $\varphi_h \in V_h$,
\[\bar{p}(\varphi_{h,x}) := \left\{ \begin{array}{cc} \frac1\epsilon & \text{if } |\varphi_{h,x}|^2 \ge 4 - \epsilon \\
p(\varphi_x) & \text{otherwise} \\
 \end{array} \right.,
 \quad \bar{q}(\varphi_{h,y}) := \left\{ \begin{array}{cc} 1 & \text{if } |\varphi_{h,y}| \le 1 \\
q(\varphi_{h,y}) & \text{otherwise}. \\
 \end{array} \right. \]
For all numerical tests, we check if the constraints above are saturated, in which case, $\varphi_h$ does not approximates a solution of Equation \eqref{eq:min surface eq} in the entire domain $\Omega$.

\subsection{Minimal surface}
The domain is a rectangle $\Omega = (0,L) \times (0,H)$, where $L = 2$ and $H = 1$.
The boundary of the domain is folded by an angle $\alpha$ along the segment $[AC]$ as sketched in Figure \ref{fig:saddle BC}.
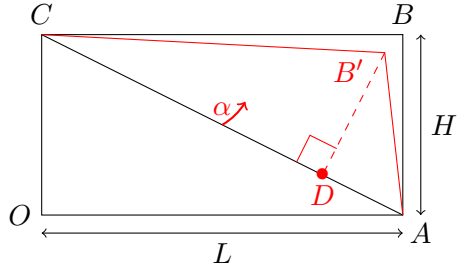
\begin{figure} [!htp]
\centering
\begin{tikzpicture}[scale=1.2]
\pgfmathsetmacro{\H}{2}
\pgfmathsetmacro{\L}{4}
\pgfmathsetmacro{\l}{0.2}

\coordinate (o) at (0,0);
\draw (o) node[left] {$O$};
\coordinate (a) at (\L,0);
\draw (\L+\l,-\l) node {$A$};
\coordinate (b) at (\L,\H);
\draw (b) node[above] {$B$};
\coordinate (c) at (0,\H);
\draw (c) node[above] {$C$};

\draw (o) -- (a) -- (b) -- (c) -- cycle;
\draw (a) -- (c);

\draw [<->,below] (0,-\l) -- (\L,-\l);
\draw (\L/2,-\l) node[below] {$L$};
\draw [<->] (\L+\l,\H) -- (\L+\l,0);
\draw (\L+\l,\H/2) node[right] {$H$};

\coordinate (bp) at (\L-\l,\H-\l);
\draw (\L-3*\l,\H-2*\l) node[red] {$B'$};
\draw[red] (c) -- (bp) -- (a);
\draw[thick,red,->] (\L/2,\H/2) arc (-60:-30:0.7cm);
\draw (\L/2,\H/2) node[above,red] {$\alpha$};

\coordinate (d) at (\L/2+\L/3.6,\H/2-\H/3.6);
\draw (d) node[red] {$\bullet$};
\draw (d) node[red,below] {$D$};
\draw[dashed,red] (bp) -- (d);
\draw [red,right angle symbol={a}{c}{bp}];
\end{tikzpicture}
\caption{Saddle shape: setup.}
\label{fig:saddle BC}
\end{figure}
Therefore, on the lines of equation $x=0$ and $y=0$, the imposed boundary conditions is $\varphi_{D1}(x,y) = \left(x,y,0 \right)^{\mathsf{T}}$,
whereas on the lines of equation $x=L$ and $x=H$, the imposed boundary condition is $\varphi_{D2}(x,y) = (1 - \frac{x}{L}) \overrightarrow{B'C} + (1 - \frac{y}{H})\overrightarrow{B'A} + \overrightarrow{OB'}$, where
\[ \overrightarrow{DB'} = BD\sin(\alpha)(0,0,1)^\mathsf{T} + \cos(\alpha)\overrightarrow{DB}. \]
Figure \ref{fig:saddle} shows the surface computed with a structured mesh of size $h=5.00\cdot 10^{-2}$ and $15,138$ dofs.
\begin{figure}[!htp]
\centering
	\subfloat{
   \includegraphics[scale=0.2]{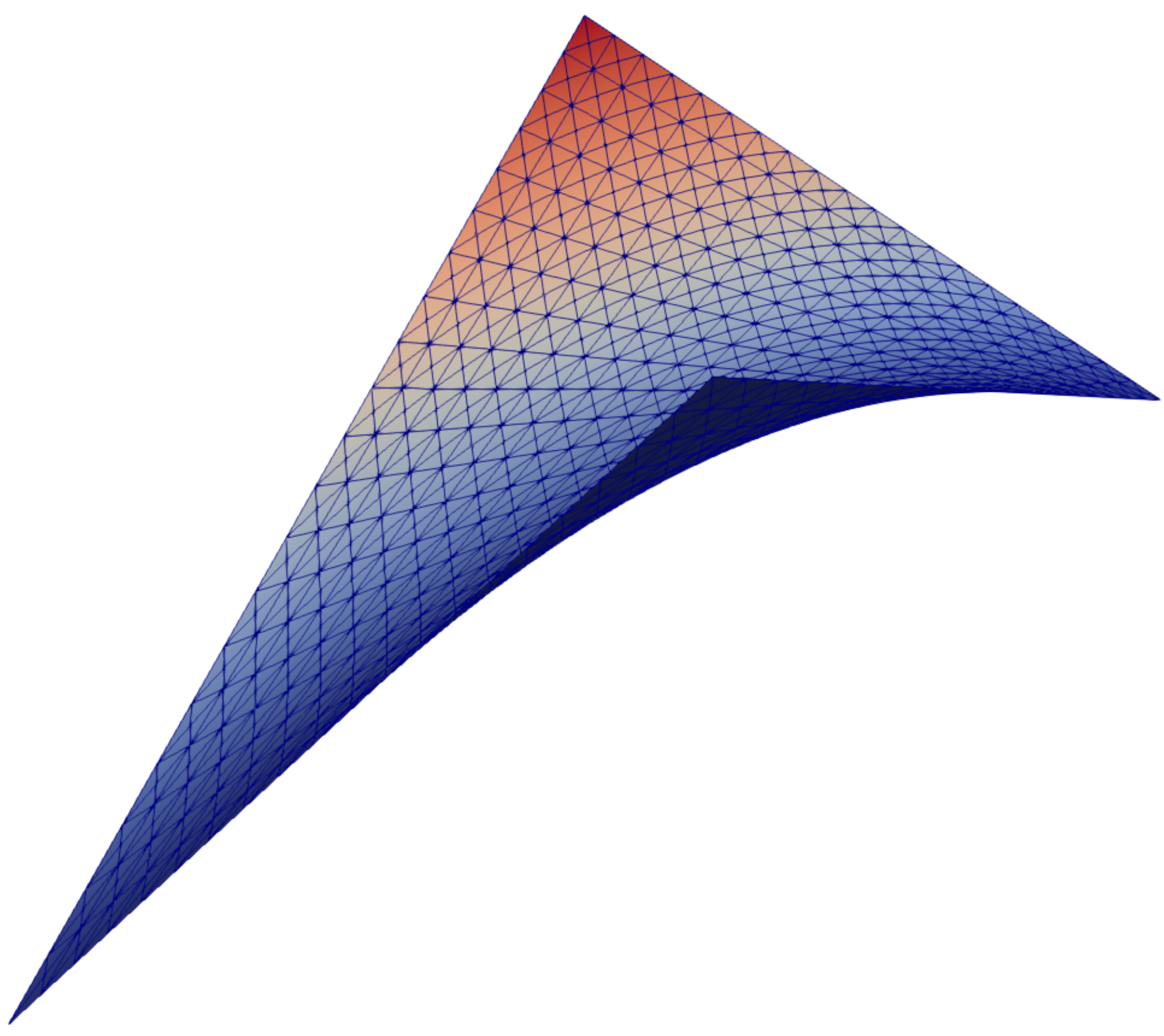}
   }
   \subfloat{
   \includegraphics[scale=0.2]{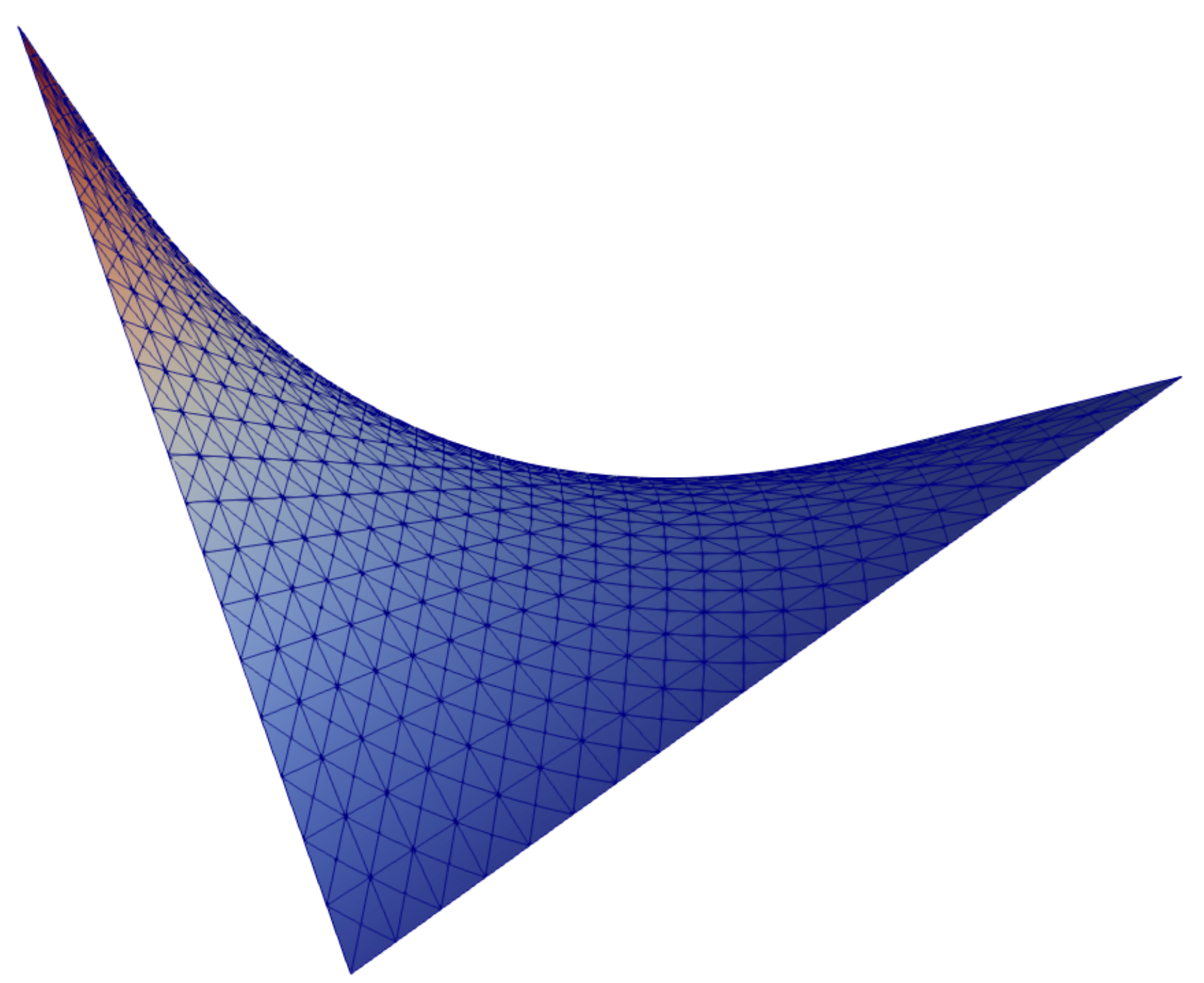}
   }
   \caption{Saddle shaped surface: resulting shape.}
   \label{fig:saddle}
\end{figure}
We note that the computation stops after the first iteration.
That suggests that the initial guess, is actually the solution and thus the solution is a minimal surface.
That is confirmed by the fact that $\max_\Omega |\Delta \varphi_h| \sim 10^{-9}$.
A closer inspection shows that actually $|\varphi_{h,xx}|\sim 10^{-9}$ and $|\varphi_{h,yy}|\sim 10^{-9}$.
One can even deduce in that case that the analytical solution is
\[\varphi(x,y) = \frac1{LH} \left( \overrightarrow{OB'}  -(L,0,0)^\mathsf{T} -(0,H,0)^\mathsf{T} \right)xy + (1,0,0)^\mathsf{T}x + (0,1,0)^\mathsf{T}y. \]
As $\varphi \in V_h$, $\varphi_h = \varphi$ on any mesh.
Also note that, here one approximates a solution of $\bar{A}(\varphi)\varphi = 0$ but not of $A(\varphi)\varphi = 0$, as $|\varphi_{h,y}| < 1$ in all of $\Omega$.
This is illustrated by Figure \ref{fig:saddle grad_y}.
\begin{figure}[!htp]
\centering
\includegraphics[scale=0.2]{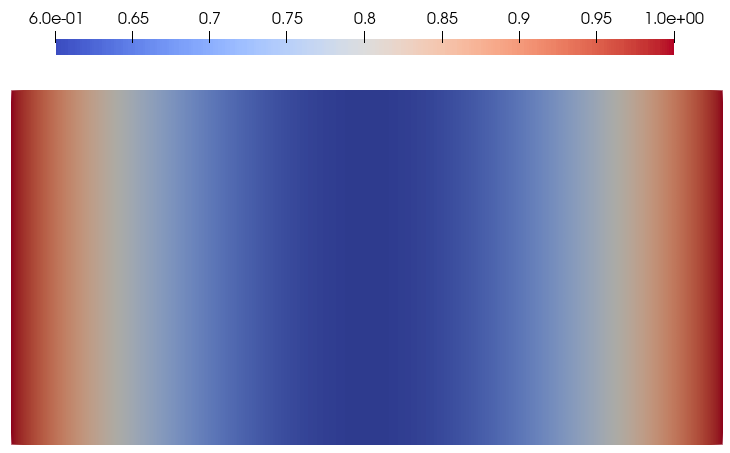}
\caption{Saddle shaped surface: $|\varphi_{h,y}|$.}
\label{fig:saddle grad_y}
\end{figure}

\subsection{Axisymmetric surface}
\label{ex:hyper}
This test case comes from \cite{lebee2018fitting}.
The reference solution is
\[\phi(x,y) = \left(\rho(x)\cos(\alpha y), \rho(x) \sin(\alpha y), z(x) \right)^\mathsf{T}, \]
where 
\[\left\{ \begin{aligned}
\rho(x) &= \sqrt{4 c_0^2 x^2 + 1}, \\
z(x) &= 2 s_0 x, 
\end{aligned} \right. \]
$\alpha = \left(1 - s_0^2 \right)^{-1/2}$, $c_0 = \cos(\frac\theta{2})$, $s_0 = \sin(\frac\theta{2})$ and $\theta \in (0, \frac{2\pi}3)$.
The domain is $\Omega = (-s_0^*,s_0^*) \times (0, \frac{2\pi}\alpha)$, where $s_0^* = \sin\left(\frac12\cos^{-1}\left(\frac1{2\cos\left(\frac{\theta}2\right)}\right)\right)$.	
Note that $\sup_\Omega |\phi_x|^2 = 2.45 < 4$.

A convergence test is performed to show that the convergence rate proved in Theorem \ref{th:convergence rate} is correct.
Let $\theta = \frac\pi{2}$.
The reference solution $\phi$ is used as Dirichlet boundary condition on $\partial \Omega$.
Table \ref{tab:convergence rate} contains the errors and estimated convergence rate.
\begin{table}[!htp]
\centering
\begin{tabular}{|c|c|c|c|c|}
\hline
$h$ & nb dofs & $H^2$-error & convergence rate & nb iterations \\ \hline
0.140 & $9,450$ & 1.815e-04 & - & 3 \\ \hline
0.0702 & $35,838$ & 2.267e-05 & 3.13 & 3 \\ \hline
0.0351 & $139,482$ & 3.006e-06 & 2.97 & 3 \\ \hline
0.0175 & $550,242$ & 3.797e-07 & 3.01 & 3 \\ \hline
\end{tabular}
\caption{Axisymmetric surface: estimated convergence rate and number of Newton iterations.}
\label{tab:convergence rate}
\end{table}
The convergence rate is estimated using the formula \[ \log\left(\frac{e_1}{e_2}\right)  \log\left(\frac{\mathrm{card}(\mathcal{T}_{h_1})}{\mathrm{card}(\mathcal{T}_{h_2})}\right)^{-1}, \]
where $e_1$ and $e_2$ are the errors in $H^2$ semi-norm.
The convergence rates presented in Table \ref{tab:convergence rate} are well above the fist order rate proved in \eqref{eq:convergence rate}. This is due to the fact that $\phi$ is far more regular than in the general case.
Indeed, $\phi \in \mathcal{C}^\infty(\Omega)^3$.
In that case, the classical interpolation result \cite{brenner2008mathematical} is $\Vert \mathcal{I}_h \phi - \phi \Vert_{H^2(\Omega)} \le C h^3 |\phi|_{H^{5}(\Omega)}$ and we recover a convergence order of $3$ as estimated in Table \ref{tab:convergence rate}.

Further computations are performed with more realistic boundary conditions.
Mirror boundary conditions are imposed on the lines of equation $y=0$ and $y=\frac{2\pi}\alpha$ of $\Omega$, which translates into the fact that the dofs on the two planes are one and the same and not doubled, but still unknown.
The Dirichlet boundary condition imposed on the lines of equations $x=-s_0^*$ and $x=s_0^*$ of are then only a circle centered around the $z$ axis, of radius $\sqrt{4c_0^2(s_0^*)^2+1}$ and contained in the planes of equations $z=\pm2s_0s_0^*$.
Figure \ref{fig:hyperboloid} shows the computed surface for $\theta = \frac{\pi}2$ and $\theta = \frac{\pi}4$.
\begin{figure}[!htp]
\centering
	\subfloat{
   \includegraphics[scale=0.08,trim={0 0 100 0},clip]{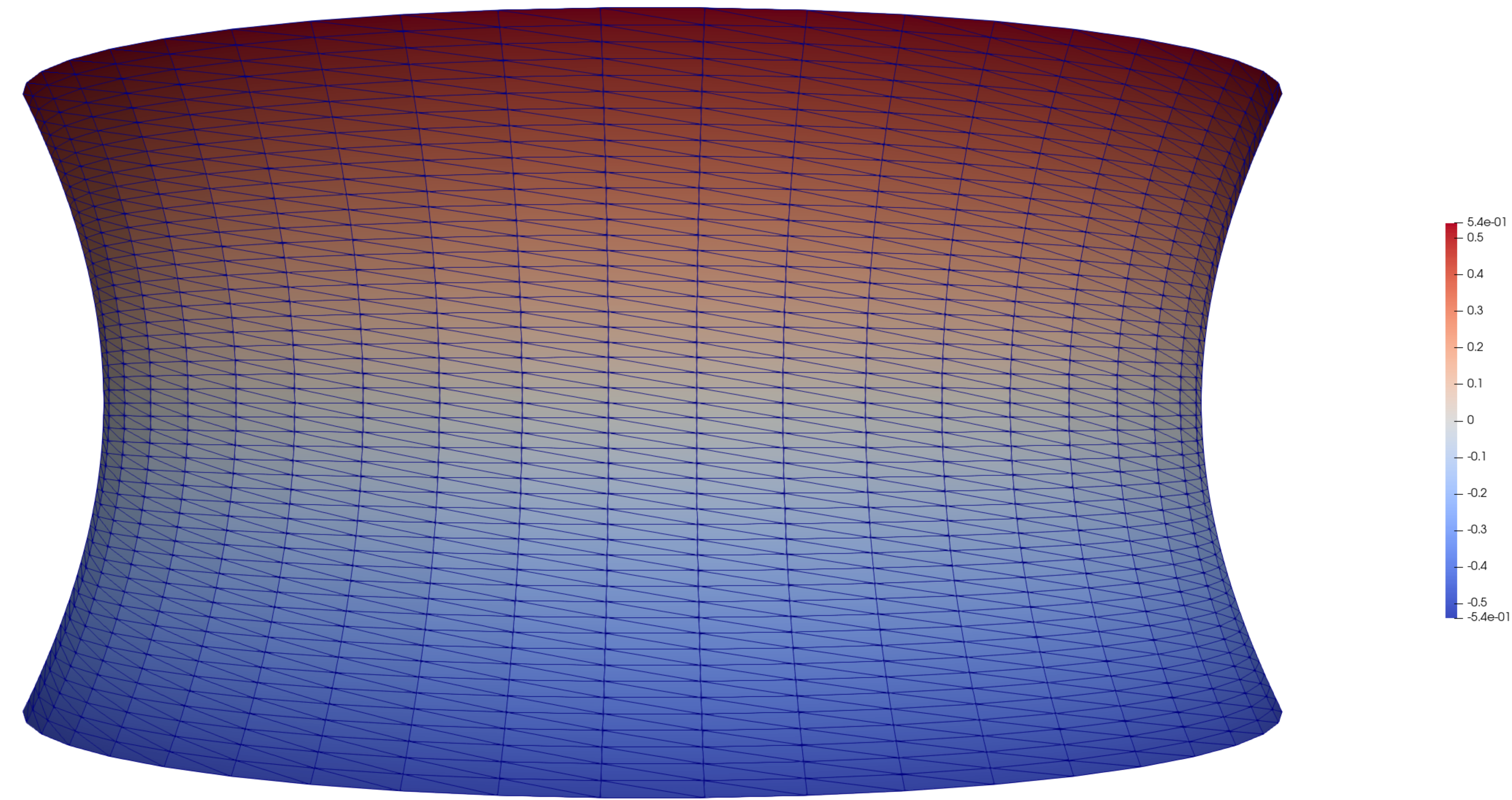}
   }
   \subfloat{
   \includegraphics[scale=0.08,trim={0 0 150 0},clip]{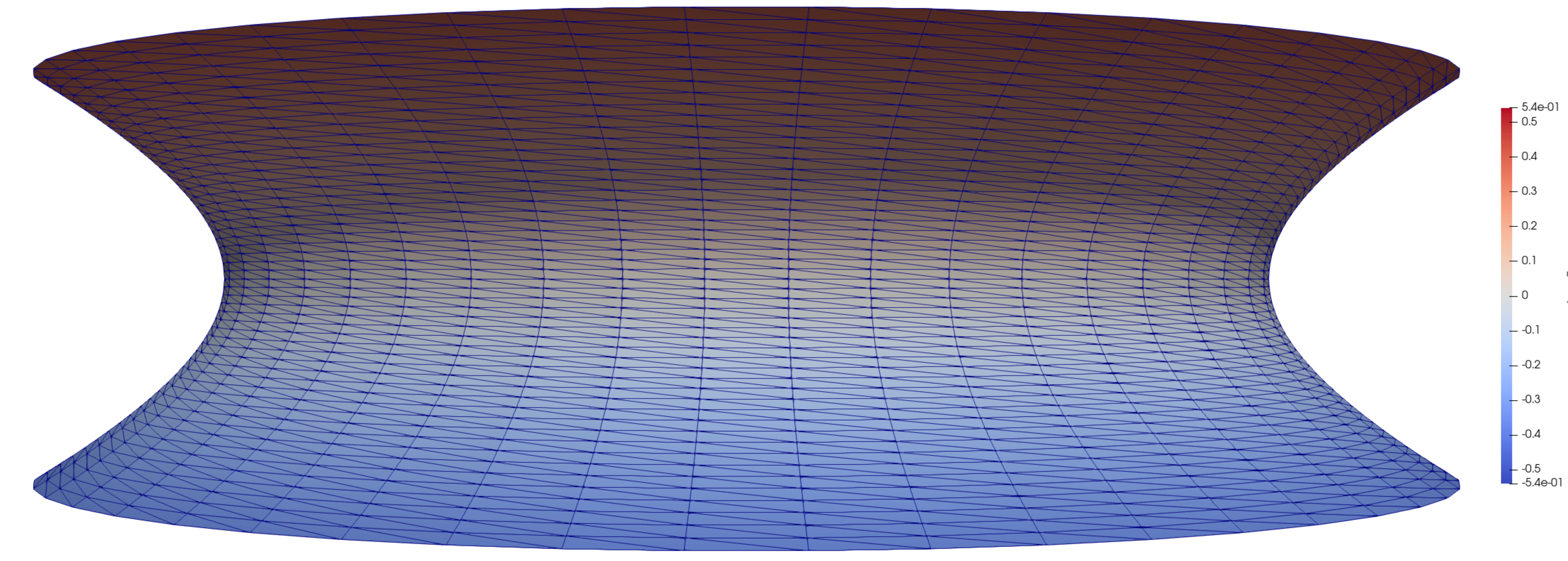}
   }
   \caption{Axisymmetric surface: Hyperboloids for $\theta = \frac{\pi}2$ (left) and for $\theta = \frac{\pi}4$ (right).}
   \label{fig:hyperboloid}
\end{figure}
We recover the two expected hyperboloids.
Note that, as suggested by the analytical solution, one has $|\varphi_{h,y} | > 1$ in all of $\Omega$.

\subsection{Non axisymmetric surface}
The boundary conditions imposed will be those of a half cone.
The initial domain is $\Omega = (0, \pi) \times (0,1)$ and the imposed Dirichlet boundary condition is
\[\varphi_D(x,y) = \left(x\cos(y), x \sin(y), x \right)^\mathsf{T}. \]
We use a structured triangular mesh of size $h=0.041$ with $59,058$ dofs.
The resulting surface is presented in Figure \ref{fig:cone} and required 4 Newton iterations.
\begin{figure}[!htp]
\centering
	\subfloat{
   \includegraphics[scale=0.15,trim={0 0 110 0},clip]{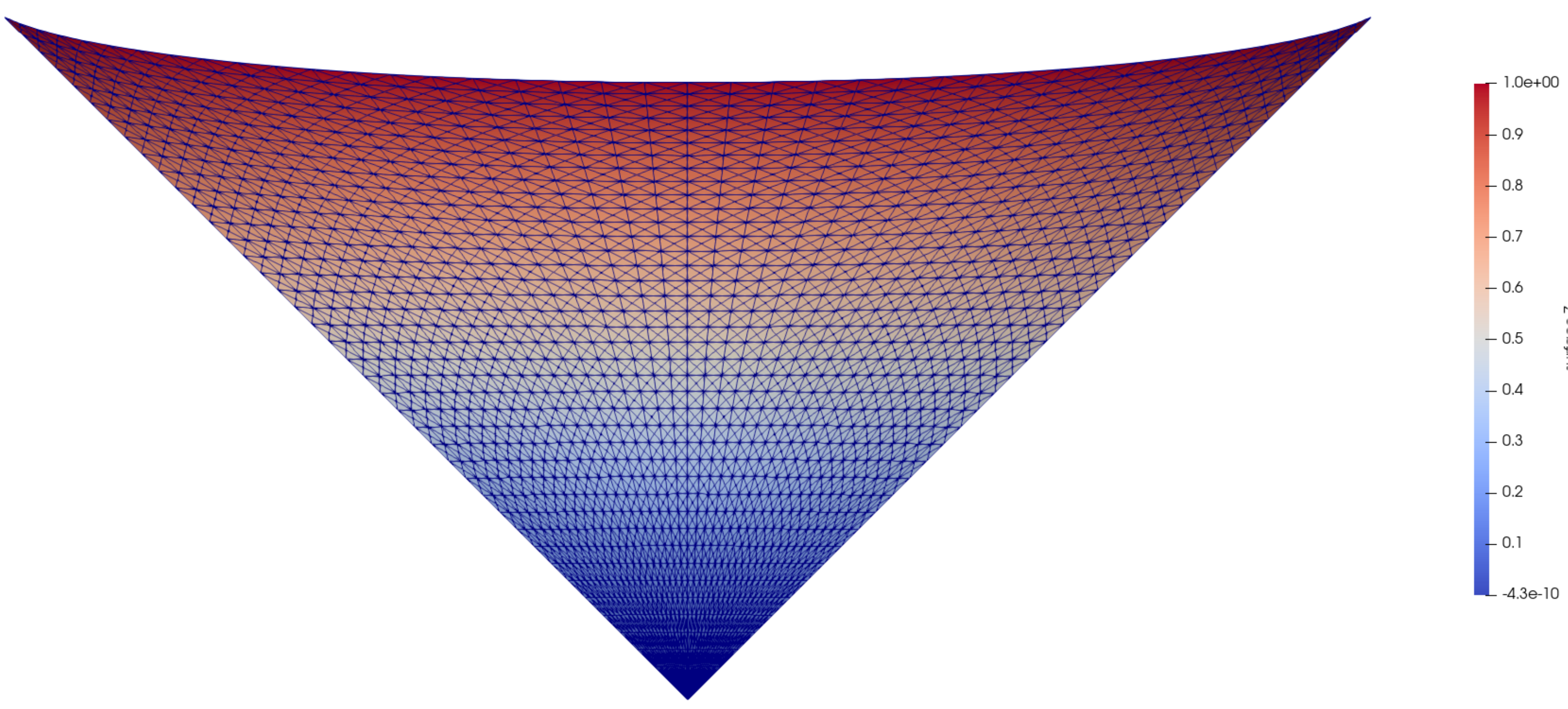}
   }
   \subfloat{
   \includegraphics[scale=0.15,trim={0 0 100 0},clip]{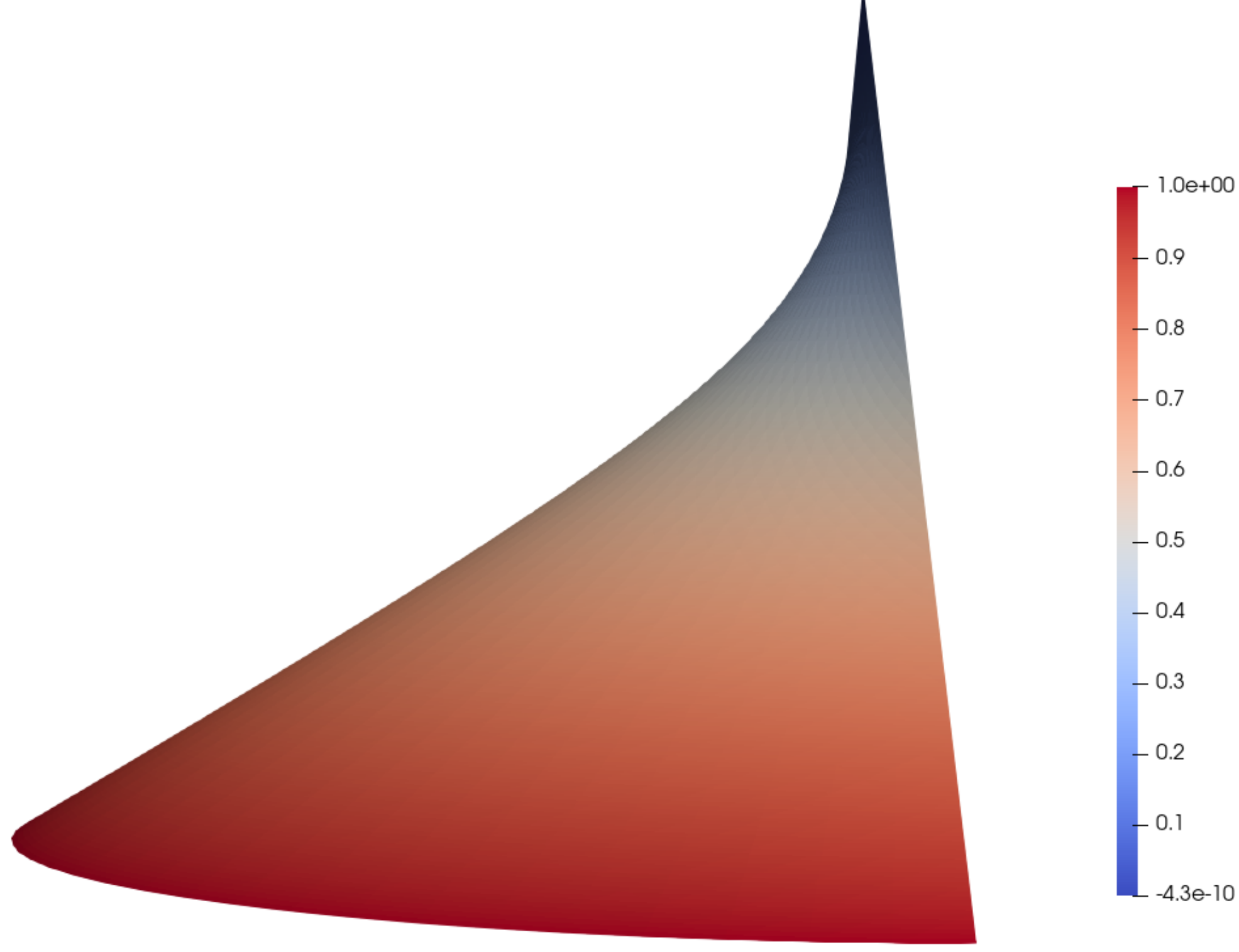}
   }   \caption{Non axisymmetric surface: computed surface}
   \label{fig:cone}
\end{figure}
Even though the left picture in Figure \ref{fig:cone} looks very much like a cone, the right picture does not.
Indeed towards the top of the surface, it tends to flatten.
Therefore, if we glue a reflexion of the surface, it will be continuous but not $\mathcal{C}^1$.
As we proved that the surface is at least $\mathcal{C}^1$, we cannot have a solution for a domain $\Omega$ larger in the $y$ component.
This is confirmed by the numerical method that stops converging for such domains.
Note that in this computation, one has $|\varphi_{h,y}| < 1$ in all of $\Omega$, as illuastred by Figure \ref{fig:cone grad_y}.
\begin{figure}[!htp]
\centering
\includegraphics[scale=0.2]{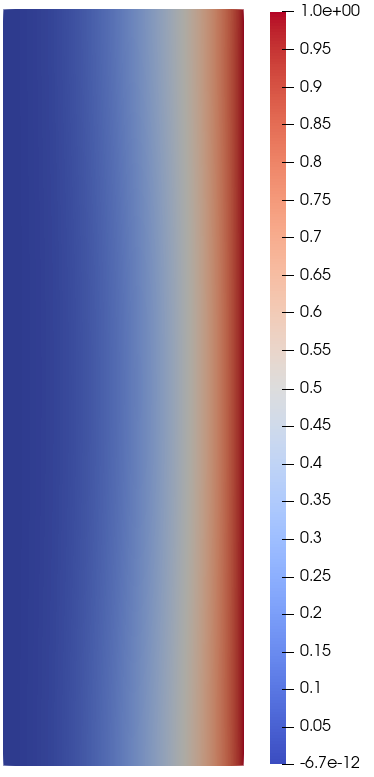}
\caption{Non axisymmetric surface surface: $|\varphi_{h,y}|$.}
\label{fig:cone grad_y}
\end{figure}
Therefore, the cone is a solution of $\bar{A}(\varphi)\varphi = 0$ and not of $A(\varphi)\varphi = 0$.
This is coherent with \cite{lebee2018fitting} which classified all the axysymmetric surfaces and did not contain a surface similar to Figure \ref{fig:cone}.

\subsection{Deformed hyperboloid}
This numerical test consists in deforming the hyperboloid of Section \ref{ex:hyper}.
The lower part of the cylinder stays unchanged whereas the upper part is slightly modified.
The domain is $\Omega = (0,L) \times (0,H)$, where $L=0.765$, $H=\frac{2\pi}\alpha$ and $\alpha=1.41$.
Periodicity is imposed on the lines of equations $y=0$ and $y=H$.
Therefore, Dirichlet boundary conditions are imposed only on the lines of equation $x=0$ and $x=L$ as
\[\left\{ \begin{aligned}
\varphi_{D1}(y) &= (R\cos(\alpha y), R\sin(\alpha y), -l)^\mathsf{T} \text{ on } x=0, \\
\varphi_{D2}(y) &= (R\cos(\alpha y), R\sin(\alpha y), R\sin(\beta)\cos(\alpha y))^\mathsf{T} \text{ on } x=L,
\end{aligned} \right. \]
where $l=1.08$, $R = 1.14$ and $\beta = \frac\pi4$.
Figure \ref{fig:weird} shows the computed surface for a structured mesh of size $h=0.222$ and containing $14,760$ dofs.
The computation requires 6 Newton iterations.
\begin{figure}[!htp]
\centering
	\subfloat{
   \includegraphics[scale=0.2]{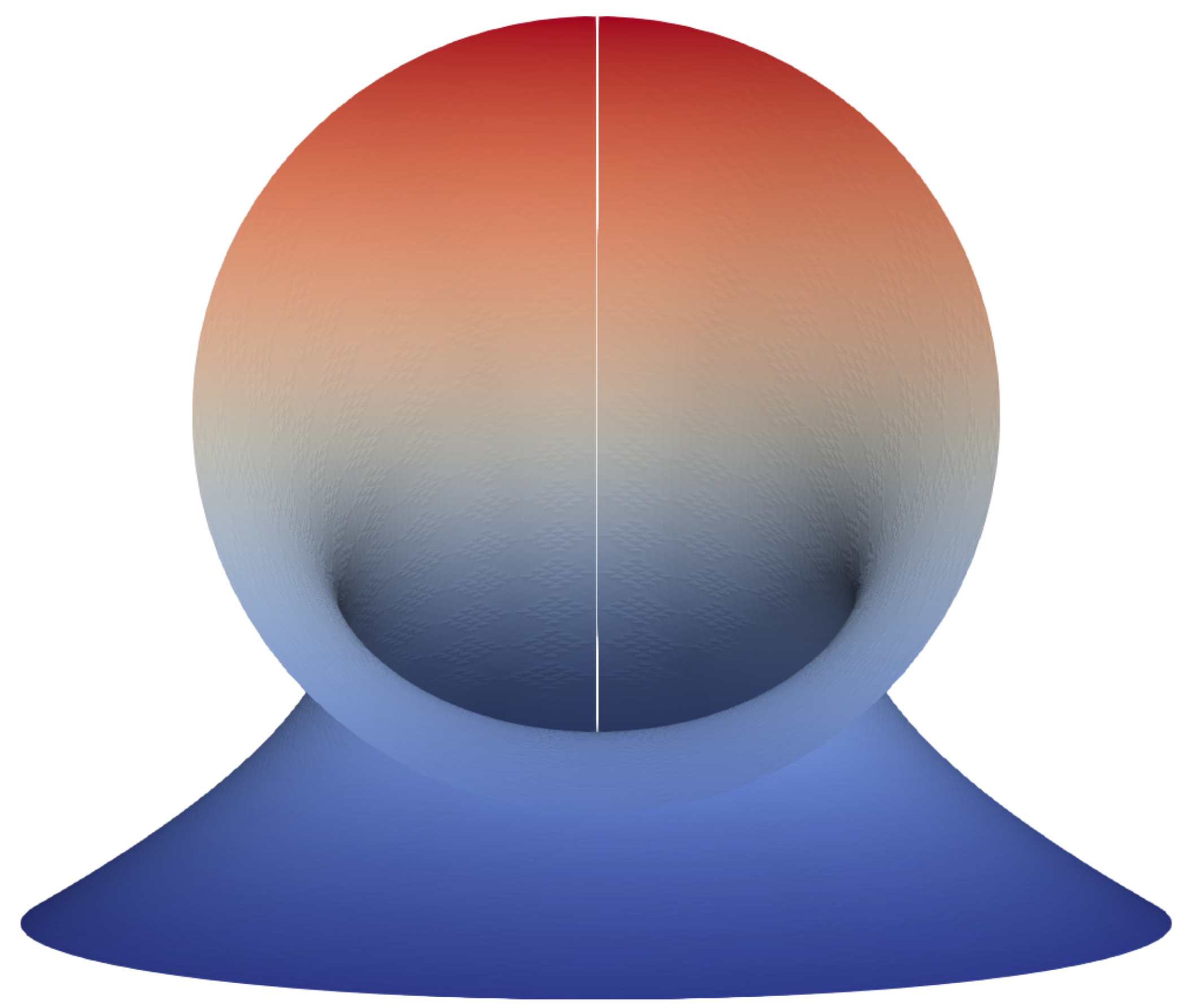}
   }
   \subfloat{
   \includegraphics[scale=0.2]{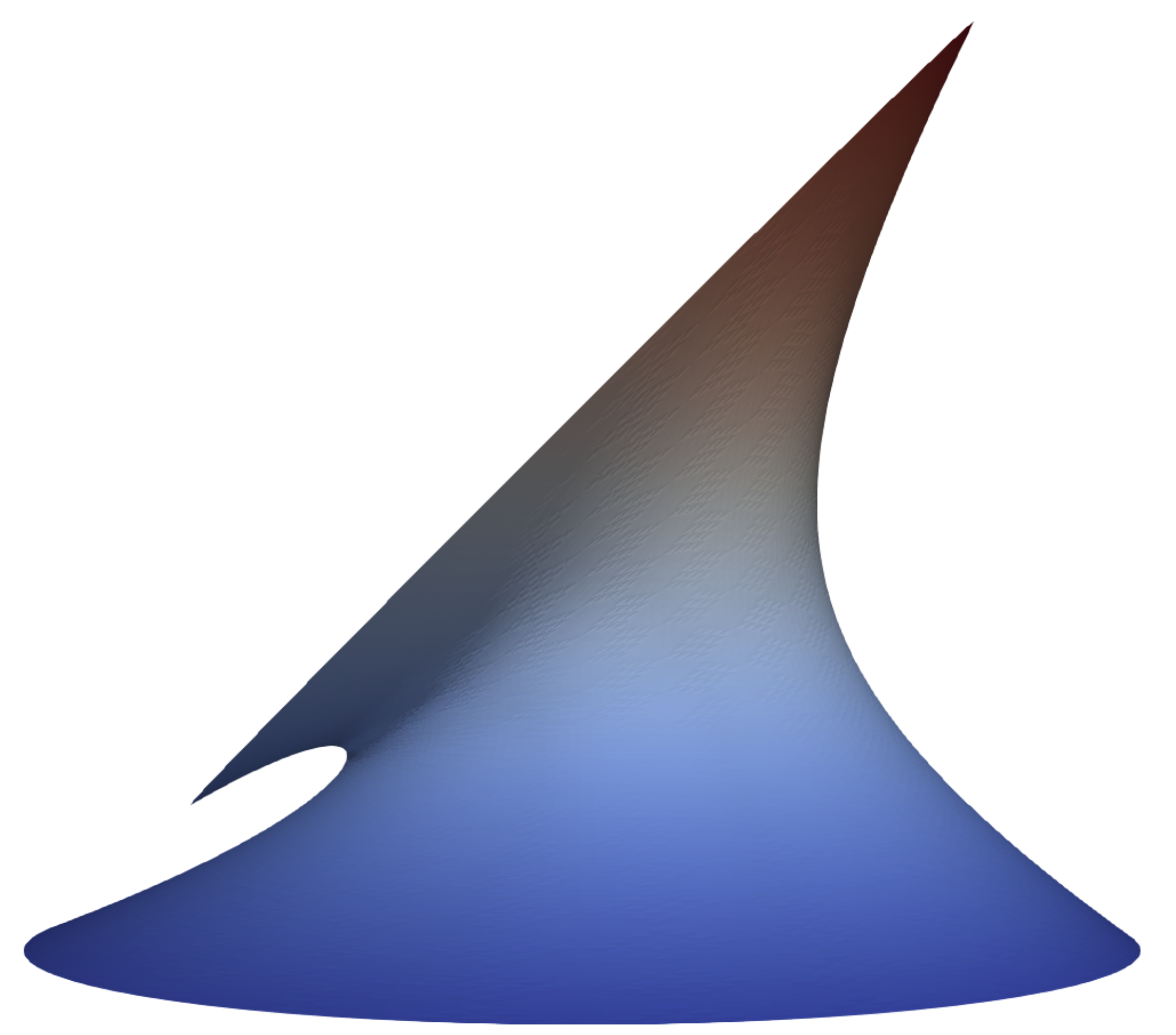}
   }
   \caption{Deformed hyperboloid: resulting shape.}
   \label{fig:weird}
\end{figure}
Note that, in all of $\Omega$, $|\varphi_{h,y}| > 1$, as illustrated by Figure \ref{fig:test grad_y}.
\begin{figure}[!htp]
\centering
\includegraphics[scale=0.21]{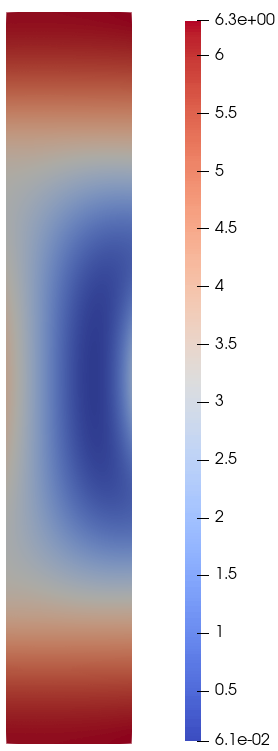}
\includegraphics[scale=0.2]{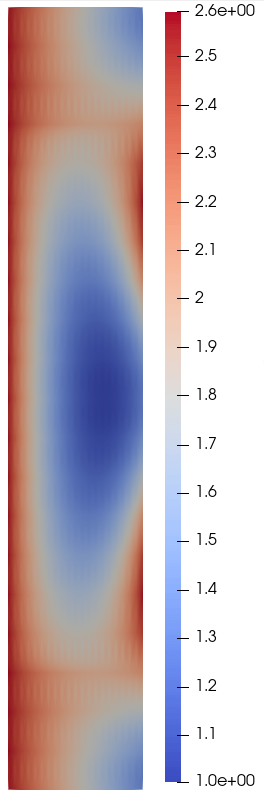}
\caption{Deformed hyperboloid: left: $|\varphi_{h,x}|^2$, right: $|\varphi_{h,y}|^2$.}
\label{fig:test grad_y}
\end{figure}
However, $\sup_\Omega | \varphi_{x,h}|^2 > 4$, and thus, $\varphi$ is a solution of $A(\varphi)\varphi = 0$ only in $\Omega' \subsetneq \Omega$.
Such a non-trivial Miura surface could not be computed with previous methods not relying on solving \eqref{eq:min surface eq}.

\section{Conclusion}
In this paper, it was proved that, under a few assumptions on the Dirichlet boundary condition, there exists of a unique solution to \eqref{eq:nonlinear operator eq}.
Subsequently, a numerical method, based on $H^2$-conforming finite elements, was presented and proved to converge at order one towards the solution of \eqref{eq:nonlinear operator eq}.
Finally, the convergence rate was validated on a numerical example and the method was used to compute a few non-trivial and non-analytical surfaces.

A question that remains unanswered is: ``Is it possible to compute a constrained solution of \eqref{eq:min surface eq} with the constraints from \cite{lebee2018fitting}".
We have seen, numerically, that the inequality constraint $|\varphi_y| > 1$ is not automatically verified.
Further investigations into the homogenization process that produced \eqref{eq:min surface eq} could prove valuable.
Investigations into how to build a Miura tessellation of a given size $\varepsilon > 0$ from a given Miura surface seem to be a natural next step. 

\section*{Code availability}
The code is available at \url{https://github.com/marazzaf/Miura_H_2.git}

\section*{Acknowledgments}
The author would like to thank A. Tarfulea (LSU) and S. Shipman (LSU) for stimulating discussions.

\section*{Funding}
This work is supported by the US National Science Foundation under grant number OIA-1946231 and the Louisiana Board of Regents for the Louisiana Materials Design Alliance (LAMDA).

\bibliographystyle{plain}
\bibliography{bib}

\end{document}